\documentclass[12pt]{article}
\usepackage{a4wide}
\usepackage[french,english]{babel}
\usepackage[utf8]{inputenc}
\usepackage[T1]{fontenc}
\usepackage{amsthm}
\usepackage{amssymb}
\usepackage{amsmath}
\usepackage{stmaryrd}
\usepackage{makeidx}
\usepackage{dsfont}
\usepackage{graphicx}
\usepackage{caption}
\theoremstyle{plain}
\usepackage[colorlinks=true,breaklinks=true,linkcolor=black]{hyperref} 
\usepackage{amsfonts}

\newtheorem{theorem}{Theorem} 
\newtheorem{prop}[theorem]{Proposition}
\newtheorem{cor}[theorem]{Corollary}
\newtheorem{definition}{Definition}
\newtheorem{lemma}[theorem]{Lemma}

\newtheorem{remark}{Remark}
\newtheorem{defn/prop}[theorem]{Definition/Proposition}
\date{}
\makeindex
\usepackage{mathrsfs}
\usepackage[all]{xy}
\theoremstyle{definition}
\usepackage[left,modulo]{lineno}
\usepackage{geometry}
\usepackage{comment}

\title{Distances on a masure}
\author{Auguste \textsc{Hébert} \\Université de Lorraine, Institut Élie Cartan de Lorraine, F-54000 Nancy, France\\ UMR 7502,
auguste.hebert@univ-lorraine.fr}

\hypersetup{
pdfauthor={Hebert, Auguste},
pdftitle={Ditances on a masure},
}

\theoremstyle{definition}\newtheorem{thm*}{Theorem}
\theoremstyle{definition}\newtheorem{proposition*}{Proposition}
\makeatletter \@addtoreset{figure}{section}\makeatother

\newcommand{\R}{\mathbb{R}}
\newcommand{\A}{\mathbb{A}}

\newcommand{\N}{{\mathbb{\Z}_{\geq 0}}}
\newcommand{\Z}{\mathbb{Z}}

\newcommand{\Ne}{{\mathbb{Z}_{\geq 1}}}

\newcommand{\I}{\mathcal{I}}
\newcommand{\T}{\mathcal{T}}
\newcommand{\Id}{\mathrm{Id}}
\newcommand{\q}{\mathfrak{q}}
\newcommand{\s}{\mathfrak{s}}
\newcommand{\F}{\mathfrak{F}}

\newcommand{\D}{\mathcal{D}}

\newcommand{\FCC}{\mathscr{F}}
\newcommand{\EC}{\mathcal{E}}

\begin{document}
\maketitle

\begin{abstract}
A masure (also known as an affine ordered hovel) $\I$ is a generalization of the Bruhat-Tits building that is associated to a split Kac-Moody group $G$ over a nonarchimedean local field. This is a union of affine spaces called apartments. When $G$ is a reductive group, $\I$ is a building and there is a $G$-invariant distance inducing a norm on each apartment. In this paper, we study distances on $\I$ inducing the affine topology on each apartment. We construct distances such that each element of $G$ is a continuous automorphism of $\I$ and we study their properties (completeness, local compactness, ...).
\end{abstract}

\tableofcontents

\section{Introduction}

If $G$ is a split Kac-Moody group over a nonarchimedean local field, St{\'e}phane Gaussent and Guy Rousseau introduced a space $\I$ on which $G$ acts  and they called this set a ``masure'' (or an ``affine ordered hovel''), see \cite{gaussent2008kac}, \cite{rousseau2017almost}.  This construction generalizes the construction of the Bruhat-Tits building associated to a split reductive group over a field equipped with a nonarchimedean valuation made by Fran\c{c}ois Bruhat and Jacques Tits, see \cite{bruhat1972groupes} and \cite{bruhat1984groupes}. A masure is an object similar to a building. It is a union of subsets called ``apartments'', each one having a structure of  a finite dimensional real-affine space and an additional structure defined by hyperplanes (called walls) of this affine space.
 The group $G$ acts  transitively on the set of apartments. It induces affine maps on each apartment, sending walls on walls. We can also define sectors and retractions from $\I$ onto apartments with center a sector-germ, as in the case of Bruhat-Tits buildings. However there can be two points of $\I$ which do not belong to a common  apartment. Studying $\I$ enables one to get information on $G$ and this is one reason to study masures.

In this paper, we assume the valuation of the valued field to be discrete. Each Bruhat-Tits building $BT$ associated to a split reductive group $H$ over a   field equipped with a discrete nonarchimedean valuation is equipped with a distance $d$ such that $H$ acts isometrically on $BT$ and such that the restriction of $d$ to each apartment is a euclidean distance. These distances are important tools in the study of buildings.  We will show that we cannot  equip masures which are not buildings with distances having these properties but it seems natural to ask whether we can define distances on a masure  which:  \begin{itemize} 
\item induce the topology of finite-dimensional real-affine space on each apartment,

\item are compatible with the action of $G$,
\item are compatible with retractions centered at a sector-germ.

\end{itemize} We show that under the assumption of continuity of retractions, the metric space we have is  never complete nor locally compact (see Subsection~\ref{subsecConsequence metriques et topologiques}). We show that there is no distance on $\I$ such that the restriction to each apartment is a norm.  However,  for each sector-germ $\s$ of $\I$, we construct distances having the following properties (Corollary~\ref{corTopologie induite sur les appartements}, Lemma~\ref{lemDistances restricted to apartments containing s}, Corollary~\ref{corRetractions lipschitziennes} and Theorem~\ref{thmAction de G lipschitzienne}): \begin{itemize}
\item the topology induced on each apartment is the affine topology,
\item each retraction with the center $\s$ is $1$-Lipschitz,
\item each retraction with center a sector-germ of the same sign as $\s$ is Lipschitz,
\item each $g\in G$ is Lipschitz   when we  regard  it as an automorphism of $\I$.
\end{itemize}

We call them \textit{distances  of} \textit{positive} or \textit{of negative} \textit{type}, depending on the sign of $\s$.  We prove that all distances of positive type on a masure (resp. of negative type) are equivalent, where   two distances $d_1$ and $d_2$ are said to be equivalent if there exist $k,\ell\in \R_{>0}$ such that $kd_1\leq d_2\leq \ell d_1$ (this is Theorem~\ref{thmEquivalence des distances}). We thus get a \textit{positive topology} $\mathscr{T}_+$ and a \textit{negative topology} $\mathscr{T}_-$ defined by distances of $\pm$ types. We prove (Corollary~\ref{corNon egalite de T+ et T-}) that these topologies are different when $\I$ is not a building. When $\I$ is a building these topologies  agree with the usual topology on a building (Proposition~\ref{propEquivalence des distances pour un immeuble}).

Let $\I_0$ be the $G$-orbit in $\I$  of some special vertex.  If $\I$ is not a building, $\I_0$ is not discrete for both $\mathscr{T}_-$ and $\mathscr{T}_+$. We also prove that if $\rho$ is a retraction centered at a negative (resp. positive) sector-germ, $\rho$ is not continuous for $\mathscr{T}_+$ (resp. $\mathscr{T}_-$), see Proposition~\ref{propNon egalite de T+ et T-}. For these reasons we introduce \textit{mixed distances}, which are  sums of a distance of positive type  with a distance of negative type. We then have the following  (Theorem~\ref{thmResume sur les Dtheta2}): all the mixed distances on $\I$ are equivalent; moreover, if $d$ is a mixed distance and $\I$ is equipped with $d$  then:

\begin{itemize}
\item each $g:\I\rightarrow \I \in G$ is Lipschitz,
\item each retraction centered at a sector-germ is Lipschitz,
\item the topology induced on each apartment is the affine topology,

\item the set $\I_0$ is discrete.

\end{itemize}

The topology $\mathscr{T}_m$ associated to mixed distances is the initial topology with respect to the retractions of $\I$ (see Corollary~\ref{corDescription de la topologie combinee}).

We prove that $\I$ is contractible for $\mathscr{T}_+$, $\mathscr{T}_-$ and $\mathscr{T}_m$.

\medskip

Let us explain how  to define distances of positive or negative type. Let $\A$ be the standard apartment of $\I$ and $C^v_f$ be the fundamental chamber of $\A$. Let $\s$ be a sector-germ of $\I$. After applying some $g\in G$ to $A$, we may assume $A=\A$ and that $\s$ is the germ $+\infty$ of  
$C^v_f$ (or  of $-C^v_f$ but this case is similar).  Fix a norm $|\ .\ |$ on $\A$.  For every $x\in \I$, there exists an apartment $A_x$ containing $x$ and $+\infty$ (which means that $A_x$ contains a sub-sector of $C^v_f$). For $u\in \overline{C^v_f}$, we define $x+u$ as the translate of $x$ by $u$ in $A_x$.  If $u$ is chosen to be sufficiently dominant, $x+u\in C^v_f$. Therefore, for all $x,x'\in \I$, there exist $u,u'\in C^v_f$ such that $x+u=x'+u'$. We then define $d(x,x')$ to be the minimum of the $|u|+|u'|$ for such couples $u,u'$.

We thus obtain a distance for each sector-germ and for each norm $|\ .\ |$ on $\A$.  

\medskip 

This paper is organized as follows.

In Section~\ref{sect general frameworks}, we  review basic definitions and set up the notation.

In Section~\ref{secSplitting of apartments}, we show that if $\s$ is a sector-germ of $\I$, we can write each apartment as a finite union of closed convex subsets each of which is contained in an apartment $A$ containing $\s$. The most important case for us is when $A$ contains a sector-germ adjacent to $\s$. We then  can write $A$ as the union of two half-apartments, each contained in an apartment containing $\s$.    We conclude Section~\ref{secSplitting of apartments} with a series of properties that distances on $\I$ cannot satisfy.

In Section~\ref{secConstruction of signed distances}, we construct distances of positive and negative type on $\I$. We prove that all the distances of positive type (resp. negative type ) are equivalent. We then study them.

 In Section~\ref{secDistances combinees}, we first show that when $\I$ is not a building, $\mathscr{T}_+$ and $\mathscr{T}_-$ are different. Then we define mixed distances and study their properties.
 
In Section~\ref{secContractilite de I}, we show that $\I$ is contractible for the topologies $\mathscr{T}_+$, $\mathscr{T}_-$ and $\mathscr{T}_m$.

\medskip

\textbf{Acknowledgements} I would like to thank St{\'e}phane Gaussent, Michael Kapovich and Guy Rousseau for their comments on a previous version of this paper.

\section{Masures}\label{sect general frameworks}
In this section, we review the theory of masures. We restrict our study to semi-discrete masures which are thick of finite thickness and such that there exists a group acting strongly transitively on them (we define these notions at the end of the section). These properties are satisfied by masures associated to split Kac-Moody groups over nonarchimedean local fields (see \cite{rousseau2016groupes}). To avoid introducing too much notation, we do not treat the case of almost split Kac-Moody groups (see \cite{rousseau2017almost}). By adapting Lemma~\ref{lemEnclos d'un demi-appart}, one can prove that our results remain valid in the almost split case.

  We begin by defining the standard apartment. References for this section are \cite{kac1994infinite}, Chapter 1 and 3, \cite{gaussent2008kac} Section 2 and \cite{gaussent2014spherical} Section 1.

\subsection{Root generating system}\label{subsecDonnees radicielles}
A Kac-Moody matrix (or generalized Cartan matrix) is a square matrix $C=(c_{i,j})_{i,j\in I}$ with integer coefficients, indexed by  a finite set $I$ and satisfying: 
\begin{enumerate}
\item $\forall i\in I,\ c_{i,i}=2$

\item $\forall (i,j)\in I^2|i \neq j,\ c_{i,j}\leq 0$

\item $\forall (i,j)\in I^2,\ c_{i,j}=0 \Leftrightarrow c_{j,i}=0$.
\end{enumerate}

A root generating system is a $5$-tuple $\mathcal{S}=(C,X,Y,(\alpha_i)_{i\in I},(\alpha_i^\vee)_{i\in I})$ made of a Kac-Moody matrix $C$ indexed by $I$, of two dual free $\Z$-modules $X$ (of characters) and $Y$ (of co-characters) of finite rank $\mathrm{rk}(X)$, a family $(\alpha_i)_{i\in I}$ (of simple roots) in $X$ and a family $(\alpha_i^\vee)_{i\in I}$ (of simple coroots) in $Y$. They have to satisfy the following compatibility condition: $c_{i,j}=\alpha_j(\alpha_i^\vee)$ for all $i,j\in I$. We also suppose that the family $(\alpha_i)_{i\in I}$  (resp. $(\alpha_i^\vee)_{i\in I}$) freely generates a $\Z$-submodule of $X$ (resp. of $Y$)).

 We now fix a Kac-Moody matrix $C$ and a root generating system with the matrix $C$.

Let $\A{}=Y\otimes \R$. We equip $\A$ with the topology defined by its structure of a finite-dimensional real-vector space. Every element of $X$ induces a linear form on $\A{}$. We will regard $X$ as a subset of the dual $\A{}^*$ of $\A{}$: the $\alpha_i$, $i\in I$ are viewed as linear forms on $\A{}$. For $i\in I$, we define an involution $r_i$ of $\A{}$ by $r_i(v)=v-\alpha_i(v)\alpha_i^\vee$ for all $v\in \A{}$. Its  fixed points set is $\ker \alpha_i$. The subgroup of $\mathrm{GL}(\A{})$ generated by the $r_i$, $i\in I$ is denoted by $W^v$ and is called the \textit{Weyl group} of $\mathcal S$. The system $(W^v,\{r_i|\ i\in I\})$ is a Coxeter system.

Let $Q=\bigoplus_{i\in I} \Z \alpha_i$ and $Q^\vee=\bigoplus_{i\in I}\Z\alpha_i^\vee$. The groups $Q$  and $Q^\vee$ are called the \textit{root lattice} and the \textit{coroot-lattice}.

One defines an action of the group $W^v$ on $\A{}^*$  as follows: if $x\in \A{}$, $w\in W^v$ and $\alpha\in \A{}^*$ then $(w.\alpha)(x)=\alpha(w^{-1}.x)$. Let $\Phi=\{w.\alpha_i|(w,i)\in W^v\times I\}$ be the set of \textit{real roots}. Then $\Phi\subset Q$. Let $Q^+=\{\sum_{i\in I} n_i\alpha_i^\vee|(n_i)\in \N^I\}\subset Q$, $Q^-=-Q^+$, $\Phi^+=\Phi\cap Q^+$ and $\Phi^{-}=\Phi\cap Q^-$. Then $\Phi=\Phi^+\cup \Phi^-$. The elements of $\Phi^+$ (resp. $\Phi^-$) are called the \textit{real positive roots} (resp. \textit{real negative roots}).  Let $W^a=Q^\vee\rtimes W^v\subset\mathrm{GA}(\A{})$ be the \textit{affine Weyl group} of $\mathcal{S}$, where $\mathrm{GA}(\A{}\mathrm{})$ is the group of affine automorphisms of $\A{}$.

For $\alpha$ and $k\in \R$, one sets $D(\alpha,k)=\{x\in \A|\ \alpha(x)+k=0\}$, $D^\circ(\alpha,k)=\{x\in \A|\ \alpha(x)+k>0\}$ and $M(\alpha,k)=\{x\in \A|\ \alpha(x)+k=0\}$. One also sets $D(\alpha,+\infty)=D^\circ(\alpha,+\infty)=\A$ and $M(\alpha,+\infty)=\emptyset$. A \textit{wall} (resp. a half-apartment)  of $\A{}$ is a hyperplane (resp. a half-space) of the form $M(\alpha,k)$ (resp. $D(\alpha,k)$) for some $\alpha\in \Phi$ and $k\in \R$. The wall (resp. half-apartment) is said to be a \textit{true} wall (resp. a true half-apartment) if $k\in \Z$ and a \textit{ghost} wall if $k\notin \Z$. This choice of true walls means that the apartment (or the masure) is semi-discrete.

\subsection{Vectorial faces and Tits preorder}\label{subVectorial faces and Tits preorder}

\subsubsection*{Vectorial faces} Define $C_f^v=\{v\in \A{}|\  \alpha_i(v)>0,\ \forall i\in I\}$. We call it the \textit{fundamental chamber}. For $J\subset I$, one sets $F^v(J)=\{v\in \A{}|\ \alpha_i(v)=0\ \forall i\in J,\alpha_i(v)>0\ \forall i\in J\backslash I\}$. Then the closure $\overline{C_f^v}$ of $C_f^v$ is the union of the subsets $F^v(J)$ for $J\subset I$. The \textit{positive} (resp. \textit{negative}) \textit{vectorial faces} are the sets $w.F^v(J)$ (resp. $-w.F^v(J)$) for $w\in W^v$  and $J\subset I$. A \textit{vectorial face} is either a positive vectorial face or a negative vectorial face.
 We call \textit{a} \textit{positive chamber} (resp. \textit{negative}) every cone  of the form $w.C_f^v$ for some $w\in W^v$ (resp. $-w.C_f^v$).  By Section 1.3 of \cite{rousseau2011masures}, the action of $W^v$ on the set of positive chambers is simply transitive. The \textit{Tits cone} $\mathcal T$ is defined as the convex cone $\mathcal{T}=\bigcup_{w\in W^v} w.\overline{C^v_f}$. We also consider the negative cone $-\mathcal{T}$.

\subsubsection*{Tits preorder on $\A$} One defines a $W^v$-invariant relation $\leq$ on $\A{}$ by: $x\leq y\ \Leftrightarrow\ y-x\in \mathcal{T}$. This preorder need not be a partial order. For example, if $W^v$ is finite (i.e when the Kac-Moody matrix $C$ defining $\mathcal{S}$ is a  Cartan matrix), then $\T=\A$ and thus $x\leq y$ for all $x,y\in \A$. For an  arbitrary Kac-Moody matrix $C$,  every element $x$ of  $\bigcap_{i\in I} \ker(\alpha_i)$ satisfies $0\leq x\leq 0$ and in particular when $C$ is not invertible, $\leq$ is not a partial order.

Let $x,y\in \A$ be such that $x\neq y$. The ray with the base point $x$ and containing $y$ (or an interval $(x,y]$, $(x,y)$, $\ldots$)  is called \textit{preordered} if $x\leq y$ or $y\leq x$ and \textit{generic} if $y-x\in \pm\mathring \T$, the interior of $\pm \T$.

\subsection{Metric properties of $W^v$}

In this subsection we prove that when $W^v$ is infinite there does not  exist  a $W^v$-invariant norm on $\A{}$ and we also establish a density property of the walls of $\A{}$.

Two true walls $M_1$ and $M_2$ are said to be \textit{consecutive} if they are of the form $\alpha^{-1}(\{k\})$, $\alpha^{-1}(\{k\pm 1\})$ for some $\alpha\in \Phi$ and some $k\in \Z$.

\begin{prop}\label{propNon existence d'une distance W-invariante}

\begin{enumerate}

\item Suppose that there exists a  $W^v$-invariant norm on $\A{}$. Then $W^v$ is finite.\label{itemNon existence d'une distance W-invariante}

\item Let $|\ .\ |$ be a norm on $\A{}$, $d$ be the induced distance on $\A{}$ and suppose that $W^v$ is infinite. Then for  every $\epsilon>0$ there exists a vectorial  wall  $M_0$ such that for all consecutive true  walls $M_1$ and $M_2$ parallel to $M_0$, $d(M_1,M_2)<\epsilon$.\label{itemMurs denses}
\end{enumerate} 
\end{prop}

\begin{proof}

Let $B\subset Y^{\dim \A}$ be a $\Z$-basis of $Y$. Then the map $W^v\rightarrow Y^{\dim \A}$ sending each $w$ to $w.B$ is injective. Thus if $W^v$ is infinite, $\{w.B|\ w\in W^v\}$ is unbounded. Part~\ref{itemNon existence d'une distance W-invariante} follows.

Suppose that $W^v$ is infinite. Let $(\beta_n)\in \Phi_+^\N$ be an injective sequence. Let $\epsilon>0$ and $u\in C^v_f$ be such that $|u|<\epsilon$. For $n\in \N$, write $\beta_n=\sum_{i\in I}\lambda_{i,n} \alpha_i$, with $\lambda_{i,n}\in \N$ for all $(i,n)\in I\times \N$. One has \[\beta_n(u)=\sum_{i\in I}\lambda_{i,n} \alpha_i(u)\geq \big(\min_{i\in I} \alpha_i(u)\big)\sum_{i\in I} \lambda_{i,n} \rightarrow +\infty.\] Let $n\in \N$ be such that $\beta_n(u)\geq 1$ and $M_0=\beta_n^{-1}(\{0\})$. Then for all consecutive true walls $M_1$ and $M_2$ parallel to $M_0$, $d(M_1,M_2)<\epsilon$, which proves the proposition. $\square$

\end{proof}

\subsection{Filters and enclosure}\label{subFilters and enclosure map}

\subsubsection*{Filters}
A filter  on a set $\mathcal{E}$ is a nonempty set $\FCC$ of nonempty subsets of $\mathcal{E}$ such that, for all subsets $E$, $E'$ of $\mathcal{E}$, one has:\begin{itemize}
\item $E$, $E'\in \FCC$ implies $E\cap E'\in \FCC$ 

\item $E'\subset E$ and $E'\in \FCC$ implies $E\in \FCC$.

\end{itemize}

If $\mathcal{E}$ is a set and $\FCC,\FCC'$ are filters on $\mathcal{E}$, we define $\FCC\Cup \FCC'$ to be the filter $\{E\cup E'| (E,E')\in \FCC\times \FCC'\}$.

If $\FCC$ is a filter on a set $\mathcal{E}$, and $E$ is a subset of $\mathcal{E}$, one says that $\FCC$ contains $E$ if every element of $\FCC$ contains $E$. We denote it $\FCC\Supset E$. If $E$ is nonempty, the principal filter on $\EC$ associated with $E$ is the filter  $\FCC_{E,\EC}$ of subsets of $\EC$ containing $E$. 

 A filter $\FCC$ is said to be contained in another filter $\FCC'$: $\FCC\Subset  \FCC'$ (resp. in a subset $Z$ in $\mathcal{E}$: $\FCC\Subset Z$) if every set in $\FCC'$  is in $\FCC$ (resp. if $Z\in \FCC$). 
 
 These definitions of containment are inspired by the following facts. Let $\EC$ be a set, $\FCC$ be a filter on $\EC$ and $E,E'\subset \EC$. Then :\begin{itemize}
\item $E\subset E'$ if and only if $\FCC_{E,\EC}\Subset \FCC_{E',\EC}$,

\item $E\Subset \FCC$ if and only if $\FCC_{E,\EC}\Subset \FCC$,

\item $E\Supset \FCC$ if and only if $\FCC_{E,\EC}\Supset \FCC$.

 \end{itemize}
If $\FCC$ is a filter on a finite-dimensional real-affine space $\mathcal{E}$, its closure $\overline \FCC$ (resp. its convex hull) is the filter of subsets of $\mathcal{E}$ containing the closure (resp. the convex hull) of some element of $\FCC$. The \textit{support} of a filter $\FCC$ on $\mathcal{E}$ is the minimal affine subspace containing $\FCC$.

\subsubsection*{Enclosure of a filter} Let $\Delta$ be the set of all roots of the root generating system $\mathcal{S}$ defined in Chapter 1 of \cite{kac1994infinite}. We only recall that $\Delta\subset \A^*$ and that $\Delta\cap \R\Phi=\Phi$. 

 Let $\FCC$ be a filter on $\A$. The enclosure $\mathrm{cl}(\FCC)$ is the filter on $\A$ defined as follows. A set $E$ is in $\mathrm{cl}(\FCC)$ if there exists $(k_\alpha)\in (\Z\cup \{+\infty\})^\Delta$ satisfying:\[E\supset \bigcap_{\alpha\in\Delta} D(\alpha,k_\alpha)\Supset \FCC.\]

Suppose that we are in the reductive case, i.e that  $\mathcal{S}$ is associated to a Cartan matrix or equivalently that $\Phi$ is finite. Then $\Delta=\Phi$. Let $E\subset \A$ and $E'$ be the intersection of the true half-apartments containing $E$ ($E'$ is the enclosure of $E$ in the definition of \cite{bruhat1972groupes}). Then $\mathrm{cl}(E)=\FCC_{E',\A}$.

\subsection{Faces, sector-faces, chimneys and germs}\label{subFaces_sector-faces}

\subsubsection*{Sector-faces, sectors} A \textit{sector-face} $f$ of $\A$ is a set of the form $x+F^v$ for some vectorial face $F^v$ and some $x\in \A$. The point $x$ is its \textit{base point} and $F^v$ is its \textit{direction}.  The \textit{germ at infinity} $\F=germ_\infty(f)$  of $f$ is the filter composed of all the subsets  of $\A$ which contain an element of the form $x+u+F^v$, for some $u\in \overline{F^v}$. 

When $F^v$ is a vectorial chamber, one calls $f$ a \textit{sector}. The intersection of two sectors of the same direction is a sector of the same direction. A \textit{sector-germ} of $\A$ is a filter which is the germ at infinity of some sector of $\A$. We denote by $\pm \infty$ the germ of $\pm C^v_f$.

The sector-face $f$ is said to be spherical if $F^v\cap \pm\mathring{\T}$ is nonempty. A \textit{sector-panel} is a sector-face contained in a wall and spanning  it as an affine space. Sectors and sector-panels are spherical. 

 Let $\q_1$ and $\q_2$ be two sector-germs of the same sign. Let $C^v_1$, $C^v_2$ be the two vectorial chambers such that $\q_1=germ_\infty(C^v_1)$ and $\q_2=germ_\infty(C^v_2)$. We say that $\q_1$ and $\q_2$ are adjacent if $\overline{C^v_1}\cap \overline{C^v_2}$ contains some sector-panel. 
 
 Let $\q,\q'$ be two sector-germs of the same sign. A \textit{gallery} between $\q$ and $\q'$ is a sequence of sector-germs $\Gamma=(\q_1,\ldots,\q_n)$ such that $n\in \N$, $\q_1=\q$, $\q_n=\q'$ and for all $i\in \llbracket 1,n-1\rrbracket$, $\q_i$ and $\q_{i+1}$ are adjacent. The \textit{length} of $\Gamma$ is $n$. For every two sector germs $\q$ and $\q'$ of the same sign, there exists a gallery joining $\q$ and $\q'$. Indeed, let $C^v$ and $C'^v$ be the vectorial chambers such that $\q=germ_\infty(C^v)$ and $\q'=germ_\infty(C'^v)$. Let $w\in W^v$ be  such that $C'^v=w.C^v$. Let $w=r_{i_1}\ldots r_{i_k}$ be a writing of $w$, with $i_1,\ldots,i_k\in I$ . Then \[germ_\infty(C^v),germ_\infty(r_{i_1}.C^v),germ_\infty(r_{i_1}r_{i_2}.C^v),\ldots,germ_\infty(r_{i_1}\ldots  r_{i_k}.C^v)\] is a gallery from $\q$ to $\q'$.

\subsubsection*{Faces} Let $x\in \A$ and let $F^v$ be a vectorial face of $\A$. The \textit{face} $F(x,F^v)$ is the filter defined as follows: a set $E\subset \A$ is an element of $F(x,F^v)$ if, and only
 if, there exist $(k_\alpha),(k_\alpha')\in (\Z\cup \{+\infty\})^\Delta$ and a neighborhood  $\Omega$  of $x$ in $\A$ such that $E\supset \bigcap_{\alpha\in \Delta}\big(D(\alpha,k_\alpha)\cap D^\circ(\alpha,k_\alpha')\big)\supset \Omega \cap (x+F^v)$. A \textit{face} of $\A$ is a filter $F$ that can be written as $F=F(x,F^v)$, for some $x\in \A$ and some vectorial face $F^v$.

A \textit{chamber} is a face whose support is $\A$. A \textit{panel} is a face whose support is a wall.

In the reductive case (i.e when $\Phi$ is finite), we obtain the usual notion of faces: the  faces for the definition we gave are exactly the $\FCC_{F,\A}$, where $F$ is a face  of $\A$ equipped with its structure of a simplicial complex.

\subsubsection*{Chimneys}   Let $F$ be a face of $\A$ and $F^v$ be a vectorial face of $\A$. The  \textit{chimney} $\mathfrak{r}(F,F^v)$ is the filter $\mathrm{cl}(\FCC_{F+F^v,\A})$. A \textit{chimney} $\mathfrak{r}$ is a filter on $\A$ of the form $\mathfrak{r}= \mathfrak{r}(F,F^v)$ for some face $F$ and some vectorial face $F^v$. The enclosure of a sector-face is thus a chimney. The vectorial face $F^v$ is uniquely determined by $\mathfrak{r}$ (this is not necessarily the case of the face $F$) and one calls it the \textit{direction} of $\mathfrak{r}$.

Let $\mathfrak{r}$ be a chimney and $F^v$ be its direction. One says that $\mathfrak{r}$ is \textit{splayed} if $F^v$ is spherical (or equivalently if $F^v$ contains a generic ray, see Subsection~\ref{subVectorial faces and Tits preorder}). One says that $\mathfrak{r}$ is \textit{solid} if the pointwise stabilizer in $W^v$ of the direction of  the support of $\mathfrak{r}$ is finite. A splayed chimney is solid. 

Let $\mathfrak{r}=\mathfrak{r}(F,F^v)$ be a chimney. A shortening of $\mathfrak{r}$ is a chimney of the form $\mathfrak{r}(F+u,F^v)$, for some $u\in \overline{F^v}$.  The \textit{germ} \textit{at infinity} $\mathfrak{R}=germ_\infty(\mathfrak{r})$ of $\mathfrak{r}$ is the filter composed of all subsets of $\A$ which contain a shortening of $\mathfrak{r}$. A sector-germ is an example of a germ of a splayed chimney.

\subsection{Masure}\label{subsecmasure}

Let $\alpha\in \Phi$. We can write $\alpha=w.\alpha_i$ for some $i\in I$ and $w\in W^v$. Then $w.\alpha_i^\vee$ does not depend on the choice of $w$ and one denotes it $\alpha^\vee$.  An \textit{automorphism} of $\A$ is an affine bijection $\phi:\A\rightarrow \A$ stabilizing the set $\{\big(M(\alpha,k),\alpha^\vee\big)| (\alpha,k)\in \Phi\times \Z\}$. One has $W^a\subset W^v\ltimes Y\subset \mathrm{Aut}(\A)$, where $\mathrm{Aut}(\A)$ is the group of automorphisms of $\A$.

An apartment of type $\A$ is a set $A$ with a nonempty set $\mathrm{Isom}^w(\A,A)$ of bijections (called Weyl isomorphisms) such that if $f_0\in \mathrm{Isom}^w(\A,A)$ then $f\in \mathrm{Isom}^w(\A,A)$ if and only if, there exists $w\in W^a$ satisfying $f=f_0\circ w$. An isomorphism (resp. a Weyl isomorphism, a vectorially Weyl isomorphism) between two apartments $\phi:A\rightarrow A'$ is a bijection such that for every $f\in \mathrm{Isom}^w(\A,A)$ and $f'\in \mathrm{Isom}^w(\A,A')$, one has $f'\circ \phi\circ f^{-1}\in \mathrm{Aut}(\A)$ (resp. $f'\circ \phi\circ f^{-1}\in W^a$, $f'\circ \phi\circ f^{-1}\in (W^v\ltimes \A)\cap \mathrm{Aut}(\A)$).

Each apartment $A$ of type $\A$ can be equipped with  the structure of an affine space by using an isomorphism of apartments $\phi:\A\rightarrow A$.  We equip each apartment with its topology defined by its structure of a finite-dimensional real-affine space. 

We extend all the notions that are preserved by   $\mathrm{Aut(\A)}$ to each apartment. In particular, enclosures, sector-faces, faces, chimneys, germs of chimneys, ... are well defined in each apartment of type $\A$.   If $A$ is an apartment of type $\A$ and $x,y\in A$, then we denote by $[x,y]_A$ the closed segment of $A$ between $x$ and $y$. 

We say that an apartment contains  a filter if it contains at least one element of this filter. We say that a map fixes a  filter if it fixes at least one element of this  filter. 

We now give the definition of masures. These objects were introduced by Gaussent and Rousseau in \cite{gaussent2008kac} (they were initially called ``hovels''). This axiomatic definition was introduced by Rousseau in \cite{rousseau2011masures}.

\begin{definition}
A  masure of type $\A$ is a set $\mathcal{I}$ endowed with a covering $\mathcal{A}$ by subsets called apartments such that: 

(MA1) Each $A\in \mathcal{A}$ admits a structure of an apartment of type $\A$.

(MA2) If $F$ is a point, a germ of a preordered interval, a generic ray or a solid chimney in an apartment $A$ and if $A'$ is another apartment containing $F$, then $A\cap A'$ contains the enclosure $\mathrm{cl}_A(F)$ of $F$ and there exists a Weyl isomorphism from $A$ onto $A'$ fixing  $\mathrm{cl}_A(F)$.

(MA3) If $\mathfrak{R}$ is the germ of a splayed chimney and if $F$ is a face or a germ of a solid chimney, then there exists an apartment that contains $\mathfrak{R}$ and $F$.

(MA4) If two apartments $A$, $A'$ contain $\mathfrak{R}$ and $F$ as in (MA3), then there exists a Weyl isomorphism from $A$ to $A'$ fixing  $\mathrm{cl}_A(\mathfrak{R}\Cup F)$.

(MAO) If $x$, $y$ are two points contained in two apartments $A$ and $A'$, and if $x\leq_{A} y$ then the two segments $[x,y]_A$ and $[x,y]_{A'}$ are equal.
\end{definition}

We assume that there exists a group $G$ acting strongly transitively on  $\I$, which means that: \begin{itemize}
\item $G$ acts on $\I$,

\item $g.A$ is an apartment for every $g\in G$ and every apartment $A$,

\item for every $g\in G$ and every apartment $A$, the map $A\rightarrow g.A$ is an isomorphism of apartments,

\item all isomorphisms involved in the above axioms are induced by elements of $G$.
\end{itemize}

 We choose in $\I$ a ``fundamental'' apartment, that we identify with $\A$. As $G$ acts strongly transitively on $\I$,  the apartments of $\I$ are the sets $g.\A$ for $g\in G$. The stabilizer $N$ of $\A$ induces a group $\nu(N)$ of affine automorphisms of $\A$ and we assume that $\nu(N)=W^v\ltimes Y$.

All the isomorphisms that we will consider in this paper will be vectorially Weyl isomorphisms and we will say ``isomorphism'' instead of ``vectorially Weyl isomorphism''.

Throughout the paper, we will only consider masures $\I$ which are thick of finite thickness, that is masures  satisfying the following axiom:

(MAT)   for each panel $P$, the number of chambers whose closure contains $P$ is finite and greater than $2$.  

This definition coincides with the usual one when $\I$ is a building. 

An example of such a masure $\I$ is the masure associated to a split Kac-Moody group over a field equipped with a nonarchimedean discrete valuation constructed in \cite{gaussent2008kac} and in \cite{rousseau2016groupes}.

 A masure $\I$ is a building if and only if $W^v$ is finite, see \cite{rousseau2011masures} 2.2 6).

\subsection{Retractions centered at sector-germs}\label{subsecRetractions}

If $A$ and $B$ are two apartments, and $\phi:A\rightarrow B$ is an isomorphism of apartments fixing some  filter $\mathcal{X}$, one writes $\phi:A\overset{\mathcal{X}}{\rightarrow} B$. If $A$ and $B$ share a sector-germ $\s$, there exists a unique isomorphism of apartments $\phi:A \rightarrow B$ fixing $A\cap B$. Indeed, by (MA4), there exists an isomorphism $\psi:A\rightarrow B$ fixing $\s$. Let $x\in A\cap B$. By (MA4), $A\cap B$ contains the convex hull $\mathrm{Conv}(x,\s)$ in $A$ of $x$ and $\s$ and there exists an isomorphism of apartments $\psi':A\rightarrow B$ fixing $\mathrm{Conv}(x,\s)$.  Then $\psi'^{-1}\circ \psi :A\rightarrow A$ is an isomorphism of affine spaces fixing $\s$: $\psi'=\psi$. By definition $\psi'(x)=x$ and thus $\psi$ fixes $A\cap B$. The uniqueness is a consequence of the fact that the only affine morphism fixing some nonempty open set of $A$ is the identity. One denotes by $A\overset{A\cap B}{\rightarrow} B$ or by $A\overset{\s}{\rightarrow} B$ the unique isomorphism of apartments from $A$ to $B$ fixing $\s$. 

Fix a sector-germ  $\s$  of $\I$ and an apartment $A$ containing $\s$. Let $x\in \I$. By (MA3), there exists an apartment $A_x$ of $\I$ containing $x$ and $\s$. Let $\phi:A_x\overset{\s}{\rightarrow} A$ fixing $\s$. By \cite{rousseau2011masures} 2.6, $\phi(x)$ does not depend on the choices we made and thus we define $\rho_{A,\mathfrak{s}}(x)=\phi(x)$.

The map $\rho_{A,\s}$ is a retraction from $\I$ onto $A$. It only depends on $\s$ and $A$ and we call it the \textit{retraction onto $A$ centered at $\s$}. We denote by $\I\overset{\s}{\rightarrow} A$ the retraction onto $A$ fixing $\s$. We denote by $\rho_{\pm\infty}$ the retraction onto $\A$ centered at $\pm\infty$.

\subsection{Parallelism in $\I$}\label{subParallelism in I}
Let us explain briefly the notion of parallelism  in $\I$. This is done  in detail in \cite{rousseau2011masures} Section 3.

Let us begin with rays. Let $\delta$ and $\delta'$ be two generic rays in $\I$. By (MA3) and \cite{rousseau2011masures} 2.2 3) there exists an apartment $A$ containing  sub-rays of $\delta$ and $\delta'$  and we say that $\delta$ and $\delta'$ are \textit{parallel}, if these sub-rays are parallel in $A$. Parallelism is an equivalence relation. The parallelism class of a generic ray $\delta$ is denoted $\delta_\infty$ and is called its \textit{direction}.

We now review  the notion of parallelism for sector-faces. We refer to \cite{rousseau2011masures}, 3.3.4)) for the details.

\subsubsection*{Twin-building $\I^\infty$ at infinity} If $f$ and $f'$ are two spherical sector-faces in $\I$, there exists an apartment $B$ containing their germs $\F$ and $\F'$. One says that $f$ and $f'$ are parallel if $\F=germ_\infty(x+F^v)$ and $\F'=germ_\infty(y+F^v)$ for some $x,y\in B$ and for some vectorial face $F^v$ of $B$. Parallelism is an equivalence relation. The parallelism class of a sector-face germ $\F$ is denoted $\F_\infty$ and is called its \textit{direction}. 
We denote by  $\I^\infty$ the set of directions of spherical faces of $\I$. If $\s$ is a sector, all the sectors having the germ at infinity $\s$ have the same direction. We denote it $\s$ by abuse of notation.
 If $M$ is a wall of $\I$, its direction $M^\infty\subset \I^\infty$ is defined to be the set of germs at infinity $\F_\infty$ such that $\F=germ_\infty (f)$, with $f$ a spherical sector-face contained in $M$.

Let $\F_\infty\in \I^\infty$ (resp. let $\delta_\infty$ be the direction of a generic ray) and $A$ be an apartment. One says that $A$ contains $\F_\infty$ (resp. $\delta_\infty$) if $A$ contains some sector-face $f$ (resp. generic ray $\delta$) whose direction is $\F_\infty$ (resp. is $\delta_\infty$).

\begin{prop}\label{defnProp faces}
\begin{enumerate}

\item\label{itDef de x+f} Let $x\in \I$ and $\F_\infty\in \I^\infty$ (resp. $\delta_\infty$ be a generic ray direction). Then there exists a unique sector-face $x+\F_\infty$ (resp. $x+\delta_\infty$) based at $x$ and whose direction is $\F_\infty$ (resp. $\delta_\infty$). 

\item\label{itDescription de x+f} Let $A_x$ be an apartment containing $x$ and $\F_\infty$ (resp. $\delta_\infty$) (which exists by (MA3)). Let $f$ (resp. $\delta'$) be a sector-face (resp. a generic ray) of $A_x$ whose direction is $\F_\infty$ (resp. $\delta_\infty$). Then $x+\F_\infty$ (resp. $x+\delta_\infty$)  is the sector-face (resp. generic ray) of $A_x$ parallel to $f$ (resp. $\delta'$) and based at $x$.

\item\label{itAppart à l'infini} Let $B$ be an apartment containing $\F_\infty$ (resp. $\delta_\infty$). Then for all $x\in B$, $x+\F_\infty\subset B$ (resp. $x+\delta_\infty\subset B$).

\end{enumerate}

\end{prop}

\begin{proof}
The points~\ref{itDef de x+f} and~\ref{itDescription de x+f} for sector-faces are Proposition 4.7.1) of \cite{rousseau2011masures} and its proof. Point~\ref{itAppart à l'infini} is a consequence of~\ref{itDescription de x+f}. The statement for rays is analogous (see Lemma~3.2 of \cite{hebertGK}).$\square$

\end{proof}

Let $f,f'$ be sector-faces. One says that $f$ dominates $f'$ (resp. $f$ and $f'$ are opposite) if  $germ_\infty(f)=germ_\infty(x+F^v)$, $germ_\infty(f')=germ_\infty(x'+F'^v)$ for some $x,x'\in \I$ and $F^v,F'^v$ two vectorial faces of a same apartment of $\I$ such that $\overline {F^v}\supset F'^v$ (resp. such that $F'^v=-F^v$). By Proposition 3.2 2) and 3) of \cite{rousseau2011masures}, these notions extend to $\I^\infty$.

\section{Splitting of apartments}\label{secSplitting of apartments}

\subsection{Splitting of apartments in two half-apartments}\label{subsectDecoupage d'apparts en deux}
The aim of this section is to show that if $A$ is an apartment, $M$ is a wall of $A$, $\F$ is a sector-panel of $M^\infty$ and $\s$ is a sector-germ dominating $\F_\infty$, then there exist two opposite half-apartments $D_1$ and $D_2$ of $A$  such that their common wall is parallel to $M$ and such that for  both $i\in \{1,2\}$, $D_i$ and $\s$ are contained in some apartment. This is Lemma~\ref{lemDecoupage d'un appartement en 2}. This property is called ``sundial configuration'' in Section~2 of \cite{bennett2014axiomatic}. This section will enable us to show that for each choice of sign, the distances of positive types and of negative types are equivalent.

 For simplicity, we assume that $\Phi$ is reduced. This assumption can be dropped with minor changes to the next lemma.

\begin{lemma}\label{lemEnclos d'un demi-appart}
Let $\alpha\in \Phi$ and $k\in \R$. Then $\mathrm{cl}(D(\alpha,k))=\FCC_{D(\alpha,\lceil k\rceil),\A}$.
\end{lemma}

\begin{proof}
By definition of $\mathrm{cl}$, $D(\alpha,\lceil k\rceil )\in \mathrm{cl}\big(D(\alpha,k)\big)$ and hence $\mathrm{cl}\big(D(\alpha,k)\big)\Subset \FCC_{D(\alpha,\lceil k \rceil),\A}$.

 Let $E\in \mathrm{cl}(D(\alpha,k))$.  By definition, there exists $(k_\beta)\in (\Z\cup \infty)^{\Delta}$ such that $E\supset \bigcap_{\beta\in \Delta} D(\beta,k_\beta)\supset D(\alpha,k)$. Let $\beta\in \Delta\backslash \{\alpha\}$. As $\beta\notin \R_+\alpha$, $D(\beta,\ell)\nsupseteq D(\alpha,k)$ for all $\ell\in \Z$. Hence $k_\beta=+\infty$.

 As the family $\big(D(\alpha,\ell)\big)_{\ell\in \R}$ is ordered by inclusion,   $k_\alpha\geq \lceil k\rceil $. 

Therefore $\bigcap_{\beta\in \Delta}D(\beta,k_\beta)=D(\alpha,k_\alpha)\supset D(\alpha,\lceil k \rceil)$. Consequently, $\FCC_{D(\alpha,\lceil k\rceil),\A}\Subset \mathrm{cl}\big(D(\alpha,k)\big)$ and thus $\mathrm{cl}(D(\alpha,k))=\FCC_{D(\alpha,\lceil k\rceil ),\A}$. $\square$
\end{proof}

\begin{lemma}\label{lemIntersection de deux apparts contenant un demi-appart}
Let $A,B$ be two distinct apartments of $\I$ containing a half-apartment $D$. Then $A\cap B$ is a true half-apartment.
\end{lemma}

\begin{proof}
Using isomorphisms of apartments, we may assume $A=\A$. Let $\alpha\in \Phi$ and $k\in \R$ be such that $D=D(\alpha,k)$. Set $M_0=\alpha^{-1}(\{0\})$. Let $S$ be a sector of $\A$ based at $0$ and dominating some sector-panel $f\subset M_0$. Let $f'=-f$ and $\s$, $\F_\infty$ and $\F'_\infty$ be the directions of $S$, $f$ and $f'$. Let $x\in \A\cap B$. Then by Proposition~\ref{defnProp faces}~(\ref{itAppart à l'infini}), $\A\cap B\supset x+\s$ and $\A\cap B\supset  x+\F'_\infty$. As $germ_\infty(x+\s)$, $germ_\infty(x+\F'_\infty)$ are the germs of splayed chimneys, we can apply (MA4) and we get that $\A\cap B\Supset \mathrm{cl}\big(germ_\infty(x+\s) \Cup germ_\infty(x+\F'_\infty)\big)$. But 
 \[\mathrm{cl}\big(germ_\infty(x+\s)\Cup germ_\infty(x+\F'_\infty)\big)=\mathrm{cl}\bigg(\overline{\mathrm{Conv}}\big(germ_\infty(x+\s)\Cup germ_\infty(x+\F'_\infty)\big)\bigg),\] where $\overline{\mathrm{Conv}}$ denotes the closure of the convex hull. Therefore 
\[\mathrm{cl}\big(germ_\infty(x+\s)\Cup germ_\infty(x+\F'_\infty)\big)=\mathrm{cl}\big(D(\alpha,-\alpha(x)\big)=\FCC_{D(\alpha,\lceil -\alpha(x) \rceil),\A}\] (by Lemma~\ref{lemEnclos d'un demi-appart}). Thus 
$\A\cap B\supset D(\alpha,\lceil -\alpha(x) \rceil) \ni x$. Consequently,
 \[\A\cap B\supset \bigcup_{x\in \A\cap B} D(\alpha,\lceil -\alpha(x) \rceil) \supset \A\cap B.\]

Hence $\A\cap B=D(\alpha,\ell)$, where $\ell =\max_{x\in \A\cap B} \lceil -\alpha(x)\rceil\in \Z$, and the lemma follows. $\square$
\end{proof}

From now on, unless otherwise stated,  ``a half-apartment'' (resp.  ``a wall'') will implicitly refer to   ``a true half-apartment'' (resp. ``a true wall'').

\begin{lemma}\label{lemAutomorphisme d'appart fixant un mur}
Let $M$ be a wall of $\A$ and $w\in W^v\ltimes Y$ be an element fixing $M$. Then $w\in\{\Id,s\}$, where $s$ is the reflection of $W^v\ltimes Y$ with respect to $M$.
\end{lemma}

\begin{proof}
One writes $w=\tau\circ u$, with $u\in W^v$ and $\tau$ a translation of $\A$. Then $u(M)$ is a wall parallel to $M$. Let $M_0$ be the wall parallel to $M$ containing $0$. Then $u(M_0)$ is a wall parallel to $M_0$ and containing $0$: $u(M_0)=M_0$. Let $C$ be a vectorial chamber adjacent to $M_0$. Then $u(C)$ is a chamber adjacent to $C$: $u(C)\in \{C,s_0(C)\}$, where $s_0$ is the reflection of $W^v$ with respect to $M_0$. After composing $u$ with $s_0$, we may assume  that $u(C)=C$ and thus $u=\Id$ (because the action of $W^v$ on the set of chambers is simply transitive). $\square$
\end{proof}

If $A$ is an apartment and $D,D'$ are half-apartments of $A$, we say that $D$ and $D'$ are opposite if $D\cap D'$ is a wall and one says that $D$ and $D'$ have opposite directions if their walls are parallel and $D\cap D'$ is not a half-apartment.

\begin{lemma}\label{lem3 appartements s'intersectant bien}
Let $A_1$, $A_2$, $A_3$ be distinct apartments. Suppose that $A_1\cap A_2$, $A_1\cap A_3$ and $A_2\cap A_3$ are half-apartments such that $A_1\cap A_3$ and $A_2\cap A_3$ have opposite directions. Let $M$ be the wall of $A_1\cap A_3$.

\begin{enumerate}
\item\label{itPropriété du Y} One has  $A_1\cap A_2\cap A_3=M$ where $M$ is the wall of $A_1\cap A_3$, and for all  $(i,j,k)\in \{1,2,3\}^3$ such that $\{i,j,k\}=\{1,2,3\}$, $A_i\cap A_j$ and $A_i\cap A_k$ are opposite. 

\item\label{itReflexions} Let $s:A_3\rightarrow A_3$ be the reflection with respect to $M$, $\phi_1:A_3\overset{A_1\cap A_3}{\rightarrow} A_1$, $\phi_2:A_3\overset{A_2\cap A_3}{\rightarrow} A_2$ and $\phi_3:A_2\overset{A_1\cap A_2}{\rightarrow} A_1$. Then the following diagram is commutative: \[\xymatrix{ A_3\ar[d]^{\phi_2}\ar[r]^{s} & A_3\ar[d]^{\phi_1}\\ A_2\ar[r]^{\phi_3}&A_1}\]

\end{enumerate}
\end{lemma}

\begin{proof}
Point~\ref{itPropriété du Y} is a consequence of ``Propri{\'e}t{\'e} du Y'' and of its proof (Section 4.9 of \cite{rousseau2011masures}).

Let $\phi=\phi_1^{-1}\circ \phi_3\circ \phi_2:A_3\rightarrow A_3$. Then $\phi$ fixes $M$.  Let $D_1=A_2\cap A_3$, $D_2=A_1\cap A_3$ and $D_3=A_1\cap A_2$. One has $\phi_3(A_2)=A_1=D_2\cup D_3$ and thus $\phi_3(D_1)=D_2$. One has $\phi_1^{-1}(D_2)=D_2$. Thus $\phi(D_1)=D_2$. We conclude with Lemma~\ref{lemAutomorphisme d'appart fixant un mur}. $\square$
\end{proof}

\begin{lemma}\label{lemGermes de quartiers opposés}
Let $\s$, $\s'$ be two opposite sector-germs of $\I$. Then there exists a unique apartment containing $\s$ and $\s'$.
\end{lemma}

\begin{proof}
The existence is a particular case of (MA3). Let $A$ and $A'$ be apartments containing $\s\Cup \s'$. Let $x\in  A\cap A'$. Then by Proposition~\ref{defnProp faces}~(\ref{itAppart à l'infini}), $A=\bigcup_{y\in x+\s} y+\s' \subset A\cap A'$, thus $A\subset A'$ and the lemma follows by symmetry. $\square$
\end{proof}

Recall the definition of $\I^\infty$ and of the direction $M^\infty$ of a wall $M$ from Subsection~\ref{subParallelism in I}. The following lemma is similar to Proposition 2.9.1) of \cite{rousseau2011masures}. This is analogous to the \textit{sundial configuration} of Section~2 of \cite{bennett2014axiomatic}.

\begin{lemma}\label{lemDecoupage d'un appartement en 2}
Let $A$ be an apartment, $M$ be a wall of $A$ and $M^\infty$ be its direction. Let $\mathfrak{F}_\infty$ be the direction of a sector-panel of $M^\infty$ and $\s$ be a sector-germ dominating $\F_\infty$ and not contained in $A$. Then there exists a unique pair $\{D_1,D_2\}$ of half-apartments of $A$  such that:\begin{itemize}
\item $D_1$ and $D_2$ are opposite with the common wall $M'$ parallel to $M$
 
\item for all $i\in\{1,2\}$, $D_i$ and $\s$ are in some apartment $A_i$.
\end{itemize}
 
Moreover: \begin{itemize}
\item $D_1$ and $D_2$ are true half-apartments
\item such apartments $A_1$ and $A_2$ are unique and if $D$ is the half-apartment of 
 $A_1$ opposite to $D_1$, then $D\cap D_2=D_1\cap D_2=M'$  and $A_2=D_2\cup D$. 
\end{itemize}
 \end{lemma}
 
\begin{proof}
Let us first show the existence of $D_1$ and $D_2$. Let $\F'_\infty$ be the sector-panel of $M^\infty$ opposite to $\F_\infty$. Let $\s_1'$ and $\s_2'$ be the sector-germs of $A$ containing $\F'_\infty$. For $i\in \{1,2\}$, let $A_i$ be an apartment of $\I$ containing $\s_i'$ and $\s$, which exists by (MA3).
Let $i\in\{1,2\}$ and $x\in A\cap A_i$. Then by Proposition~\ref{defnProp faces}~(\ref{itAppart à l'infini}), $x+\s_i'\subset A\cap A_i$ and the open half-apartment $E_i=\bigcup_{y\in x+\s_i'} y+\F_\infty\subset A\cap A_i$ is contained in $A$ and $A_i$. 

Suppose $A_1=A_2$. Then $A_1\supset \bigcup_{x\in E_1}x+\s_2'=A$ and thus $A_1=A\Supset \s$ . This is absurd and thus $A_1\neq A_2$.

The apartments $A_1,A_2$ contain $\F'_\infty$ and $\s$. Take $x\in A_1\cap A_2$. Then by Proposition~\ref{defnProp faces}~(\ref{itAppart à l'infini}), $A_1\cap A_2$ contains the open half-apartment $\bigcup_{y\in x+\s}y+\F'_\infty$. By Lemma~\ref{lemIntersection de deux apparts contenant un demi-appart}, $A_1\cap A_2$ is a half-apartment. Thus we can apply Lemma~\ref{lem3 appartements s'intersectant bien}: $A_1\cap A_2\cap A=M'$, where $M'$ is a wall of $A$ parallel to $M$. Set  $D_i=A\cap A_i$ for all $i\in \{1,2\}$ . Then $\{D_1,D_2\}$ fulfills the  requirements of the lemma.

Let $D_1'$, $D'_2$ be another pair of opposite half-apartments of $A$ such that for all $i\in \{1,2\}$, $D_i'$ and $\s$ are contained in some apartment $A_i'$ and such that $D_1'\cap D_2'$ is parallel to $M$.

We can assume $D_i'\Supset \s_i'$ for both $i\in \{1,2\}$. Let $\s'$ be the sector-germ of $A'_i$ opposite to $\s$. Then $\s'$ dominates $\F'_\infty$ and is contained in $D_i'$. Therefore $\s'=\s_i$. By Lemma~\ref{lemGermes de quartiers opposés}, $A_i'=A_i$, which proves the uniqueness of $\{D_1,D_2\}$ and $\{A_1,A_2\}$.

Moreover, by Proposition 2.9 2) of \cite{rousseau2011masures}, $D\cup D_2$ is an apartment. As $D\cup D_2 \Supset \s\Cup \s_2$, one has $D\cup D_2=A_2$, which concludes the proof of the lemma.  $\square$
\end{proof}

\subsection{Splitting of apartments}

In this subsection we mainly generalize Lemma~\ref{lemDecoupage d'un appartement en 2}. We show that if $\s$ is a sector-germ of $\I$ and if $A$ is an apartment of $\I$, then $A$ is the union of a finite number of convex closed subsets $P_i$ of $A$ such that for all $i$, $P_i$ and $\s$ are contained in some apartment. This is Proposition~\ref{lemDecoupages des apparts}.
\medskip

Let $\s,\s'$ be two sector-germs of the same sign. Let $A$ be an apartment containing $\s$ and $\s'$, which exists by (MA3). Let $d(\s,\s')$ be the length of a minimal gallery from $\s$ to $\s'$ (see Subsection~\ref{subFaces_sector-faces} for the definition of a gallery). By (MA4), $d(\s,\s')$ does not depend on the choice of $A$. 

Let $\s$ be a sector-germ and $A$ be an apartment of $\I$. Let $d_\s(A)$ be the minimum of the $d(\s,\s')$, where $\s'$ runs over the sector-germs of $A$ of the same sign as $\s$.
 Let $\D_A$ be the set of half-apartments of $A$. One sets $\mathcal{P}_{A,0}=\{A\}$ and for all $n\in \Ne$, $\mathcal{P}_{A,n}=\{\bigcap_{i=1}^nD_i|(D_i)\in (\D_A)^n\}$. 
The following proposition is very similar to Proposition 4.3.1 of \cite{charignon2010immeubles}.

\begin{prop}\label{lemDecoupages des apparts}
Let $A$ be an apartment of $\I$, $\s$ be a sector-germ of $\I$ et $n=d_\s(A)$. Then there exist $P_1,\ldots,P_k\in \mathcal{P}_{A,n}$, with $k\leq 2^n$ such that $A=\bigcup_{i=1}^k P_i$ and for each $i\in \llbracket 1,k \rrbracket$, $P_i$ and $\s$
 are contained in some apartment $A_i$ such that there exists an isomorphism $f_i:A_i\overset{P_i}{\rightarrow} A$.
\end{prop}

\begin{proof}
 We prove the proposition by the induction on $n$. This is clear if $n=0$. Let $n\in \N_{>0}$. Suppose this is true for every apartment $B$ such that $d_\s(B)\leq n-1$.

  Let $B$ be an apartment such that $d_\s(B)=n$. Let $\mathfrak{t}$ be a sector-germ of $B$ such that there exists a minimal gallery $\mathfrak{t}=\s_0,\ldots,\s_{n-1}=\mathfrak{s}$ from $\mathfrak{t}$ to $\mathfrak{s}$. By Lemma~\ref{lemDecoupage d'un appartement en 2}, there exist opposite half-apartments $D_1,D_2$ of $B$ such that for both $i\in \{1,2\}$, $D_i$ and $\s_1$ are contained in an apartment $B_i$. Let $i\in\{1,2\}$. One has $d_\s(B_i)=n-1$ and thus $B_i=\bigcup_{j=1}^{k_i} P_j^{(i)}$, with $k_i\leq 2^{n-1}$, for all $j\in \llbracket 1,k_i\rrbracket$,  $P_j^{(i)}\in \mathcal{P}_{B_i,n-1}$ and $\s,P_j^{(i)}$ is contained in some apartment $A_j^{(i)}$. One has \[B=D_1\cup D_2=B_1\cap D_1\cup B_2\cap D_2=\bigcup_{i\in \{1,2\},j\in \llbracket 1,k_i\rrbracket }P_j^{(i)}\cap D_i.\]
 
Let $i\in \{1,2\}$, $j\in \llbracket 1,k_i\rrbracket$ and $\phi_i:B_i\overset{B\cap B_i}{\rightarrow} B$.  Then $P_j^{(i)}\cap D_i=\phi_i(P_j^{(i)}\cap D_i)\in \mathcal{P}_{B,n}$ and $B_i\supset (P_j^{(i)}\cap D_i),\s$.

Let $f_i^{(j)}:A_i^{(j)}\overset{P_i^{(j)}}{\rightarrow} B_i$ and $f=\phi_i\circ f_i^{(j)}$. Then $f:A_j^{(i)}\overset{P_j^{(i)}\cap D_i}{\rightarrow} B$ and the proposition follows. $\square$
\end{proof}

We deduce from the previous proposition a corollary which was already known for masures associated to split Kac-Moody groups over fields equipped with a nonarchimedean discrete valuation by Section 4.4 of \cite{gaussent2008kac}:

\begin{cor}\label{corSegments germes de quartiers}
Let $\s$ be a sector-germ, $A$ be an apartment and $x,y\in A$. Then there exists $x=x_1,\ldots,x_k=y\in [x,y]_A$ such that $[x,y]_A=\bigcup_{i=1}^{k-1}[x_i,x_{i+1}]_A$ and such that for every $i\in \llbracket 1,k-1\rrbracket$, $\s$ and $[x_i,x_{i+1}]_A$ are contained in an apartment $A_i$ such that there exists an isomorphism $f_i:A\overset{[x_i,x_{i+1}]_{A_i}}{\rightarrow}A_i$.
\end{cor}

\subsection{Restrictions on the distances}\label{subsecConsequence metriques et topologiques}

In this subsection, we show that some properties cannot be satisfied by distances on masures. If  $A$ is an apartment of $\I$,  we show that there exist apartments branching at  every wall of $A$ (this is Lemma~\ref{lemApparts branchant partout}). This implies that if $\I$ is not a building the interior of each apartment is empty for the distances we study. We write $\I$ as a countable union of apartments and then use Baire's Theorem to show that under a rather weak assumption of regularity for retractions, a masure cannot be complete nor locally compact for the distances we study.

\medskip
Let us show a slight refinement of Corollaire 2.10 of \cite{rousseau2011masures}:

\begin{lemma}\label{lemApparts branchant partout}
Let $A$ be an apartment of $\I$ and $D$ be a half-apartment of $A$. Then there exists an apartment $B$ such that $A\cap B=D$.
\end{lemma}

\begin{proof}
Let $M$ be the wall of $D$, $P$ be a panel of $M$ and $C$ be a chamber whose closure contains $P$ and which is not contained in $A$. By Proposition~2.9 1) of \cite{rousseau2011masures}, there exists an apartment $B$ containing $D$ and $C$. By Lemma~\ref{lemIntersection de deux apparts contenant un demi-appart}, $A\cap B=D$, which proves the lemma.
 $\square$
\end{proof}

\begin{prop}\label{lemNon existence d'une distance induisant une norme sur chaque appartement}
 Assume that there exists a distance $d_\I$ on $\I$ such that for  every apartment $A$, $d_{\I}|_{A^2}$ is induced by some norm. Then $\I$ is a building and $d_{\I}|_{\A^2}$ is $W^a$-invariant.
\end{prop}

\begin{proof}

Let $\s$ be a sector-germ and $A$, $B$ be two apartments containing $\s$. Let $\phi:A\overset{A\cap B}{\rightarrow} B$. Let us first prove that  $\phi:(A,d_\I)\rightarrow (B,d_\I)$ is an isometry. Let $d':A\times A\rightarrow \R_+$ be defined by $d'(x,y)=d_\I(\phi(x),\phi(y))$ for all $x,y\in A$. Then $d'$ is induced by some norm. Moreover $d'|_{(A\cap B)^2}=d_{\I}|_{(A\cap B)^2}$. As $A\cap B$ has nonempty interior, we deduce that $d'=d_\I$ and thus $\phi:(A,d_\I)\rightarrow(B,d_\I)$ is an isometry.

Let $M$ be a wall of $\A$, $D_1$ and $D_2$ be the half-apartments defined by $M$ and $s\in W^a$ be the reflection with respect to $M$. Let $A_2$ be an apartment of $\I$ such that $A\cap A_2=D_1$, which exists by Lemma~\ref{lemApparts branchant partout}. Let $D_3$ be the half-apartment of $B$ opposite to $D_1$. Then $D_3\cap D_2\subset D_3\cap \A\subset M$ and thus $D_2\cap D_3=M$. By Proposition 2.9 2) of \cite{rousseau2011masures}, $D_3\cup D_2$ is an apartment $A_1$ of $\I$. Let $\phi_2:\A\overset{\A\cap A_1}{\rightarrow} A_1$, $\phi_1:\A\overset{A\cap A_2}{\rightarrow} A_2$ and $\phi_3:A_1\overset{A_1\cap A_2}{\rightarrow} A_2$. Then by Lemma~\ref{lem3 appartements s'intersectant bien}, the following diagram is commutative: \[\xymatrix{ \A\ar[d]^{\phi_2}\ar[r]^{s} & \A\ar[d]^{\phi_1}\\ A_1\ar[r]^{\phi_3}&A_2}.\] 

 By the first part of the proof, $s$ is an isometry of $\A$ and thus $W^a$ is a group of isometries for $d_{\I}|_{\A^2}$. By Proposition~\ref{propNon existence d'une distance W-invariante}~(\ref{itemNon existence d'une distance W-invariante}), $W^v$ is finite and by \cite{rousseau2011masures} 2.2 6), $\I$ is a building. $\square$
\end{proof}

\begin{lemma}\label{lemApparts fermes quand une retraction est continue}
Let $\s$ be a sector-germ of $\I$ and $d$ be a distance on $\I$ inducing the affine topology on each apartment and such that there exists a continuous retraction $\rho$ of $\I$ centered at $\s$. Then each apartment containing $\s$ is closed.
\end{lemma}

\begin{proof}
Let $A$ be an apartment containing $\s$ and $B=\rho(\I)$. Let $\phi:B\overset{\s}{\rightarrow}A$ and $\rho_A:\I\overset{\s}{\rightarrow} A$. Then $\rho_A=\phi\circ \rho$ is continuous because $\phi$ is an affine map. Let $(x_n)\in A^\N$ be a converging sequence for $d$ and $x=\lim x_n$. Then $x_n=\rho_A(x_n)\rightarrow \rho_A(x)$ and thus $x=\rho_A(x)\in A$. $\square$
\end{proof}

\begin{prop}\label{propApparts d'interieur vide}
Suppose $\I$ is not a building. Let $d$ be a distance on $\I$ inducing the affine topology on each apartment. Then the interior of each apartment of $\I$ is empty.
\end{prop}

\begin{proof}
Let $V$ be a nonempty open set of $\I$. Let $A$ be an apartment of $\I$ such that $A\cap V\neq \emptyset$. By Proposition~\ref{propNon existence d'une distance W-invariante}~(\ref{itemMurs denses}), there exists a wall $M$ of $A$ such that $M\cap V\neq \emptyset$. Let $D$ be a half-apartment delimited by $M$. Let $B$ be an apartment such that $A\cap B=D$, which exists by Lemma~\ref{lemApparts branchant partout}. Then $B\cap V$ is an open set of $B$ containing $M\cap V$ and thus $E\cap V\neq \emptyset$, where $E$ is the half-apartment of $B$ opposite to $D$. Therefore $V\backslash A\neq \emptyset$ and we get the proposition. $\square$
\end{proof}

One sets $\I_0=G.0$ where $0\in\A$. This is the set of \textit{vertices of type 0}. Recall that  $\pm\infty=germ_\infty(\pm C^v_f)$ and that $\rho_{\pm\infty}:\I\overset{\pm\infty}{\rightarrow} \A$.

\begin{lemma}\label{lemIntersection de I0 et A}
One has $\I_0\cap\A=Y$.
\end{lemma}

\begin{proof}
Let $\lambda\in \I_0\cap \A$. Then $\lambda=g.0$  for some $g\in G$. By (MA2), there exists $\phi:g.\A\rightarrow \A$ fixing $\lambda$. Then $\lambda=\phi(g.0)$ and $\phi\circ g|_{\A}:\A\rightarrow \A$ is an automorphism of apartments. Let $h\in G$ inducing $\phi$ on $g.\A$. Then $h.g\in N$, hence $(h.g)|_{\A}\in \nu(N)=W^v\ltimes Y$ (by the end of Subsection~\ref{subsecmasure}) and thus $\lambda=h.g.0\in Y.$ $\square$
\end{proof}

\begin{lemma}\label{I0 denombrable}
The set $\I_0$ is countable.
\end{lemma}

\begin{proof}
For all $\lambda\in \I_0$, $\rho_{-\infty}(\lambda)$, $\rho_{+\infty}(\lambda)\in \I_0$ and thus $\rho_{-\infty}(\lambda)$, $\rho_{+\infty}(\lambda)\in Y$. Therefore 
$\I_0=\bigcup_{(\lambda,\mu)\in Y^2}\rho_{-\infty}^{-1}(\{\lambda\})\cap \rho_{+\infty}^{-1}(\{\mu\})$. By Theorem~5.6 of \cite{hebertGK}, $\rho_{-\infty}^{-1}(\{\lambda\})\cap \rho_{+\infty}^{-1}(\{\mu\})$ is finite for all $(\lambda,\mu)\in Y^2$, which completes the proof. $\square$
\end{proof}

Let $\s$ be a sector-germ of $\I$. For  $\lambda\in \I_0$ choose an apartment $A(\lambda)$ containing $\lambda+\s$. Let $x\in \I$ and $A$ be an apartment containing $x$ and  $\s$. Then there exists $\lambda\in \I_0\cap A$ such that $x\in \lambda+\s$ and thus $x\in  A(\lambda)$. Therefore $\I=\bigcup_{\lambda\in \I_0}A(\lambda)$.

\begin{prop}\label{propMasure non complete dans le cas denombrable}
Let $d$ be a distance on $\I$. Suppose that there exists a sector-germ $\s$ such that  every apartment containing $\s$ is closed and with empty interior. Then $(\I,d)$ is incomplete and the interior of   every compact subset of $\I$ is empty.
\end{prop}

\begin{proof}
One has $\I=\bigcup_{\lambda\in \I_0}A(\lambda)$, with $\I_0$ countable by Lemma~\ref{I0 denombrable}. Thus by Baire's Theorem, $(\I,d)$ is incomplete. 

Let $K$ be a compact subset of $\I$. Then  $K=\bigcup_{\lambda\in \I_0} K\cap A(\lambda)$ and thus $K$ has empty interior. $\square$

\end{proof}

\section{Distances of positive type and of negative type}\label{secConstruction of signed distances}

\subsection{Translation in a direction}

Let $\s$ be a sector-germ. We now define a map $+_\s$ such that for all $x\in \I$ and $u\in \overline{C^v_f}$, $x+_\s u$ is the ``translate of $x$ by $u$ in the direction $\s$''.  Let $\mathrm{sgn}(\s)\in \{-,+\}$ be the  sign of $\s$.

\begin{defn/prop}Let $\s$ be a sector-germ. Let $x\in \I$. Let $A_1$ be an apartment containing $x+\s$. Let $(x+\overline{\s})_{A_1}$ be the closure of $x+\s$ in $A_1$. Then $(x+\overline{\s})_{A_1}$ does not depend on the choice of $A_1$ and we denote it by $x+\overline{\s}$. 
\end{defn/prop}

\begin{proof}
Let $A_2$ be an apartment containing $x+\s$ and $\phi:A_1\overset{A_1\cap A_2}{\rightarrow} A_2$.  By (MA4), $\phi$ fixes the enclosure of $x+\s$, which contains  $(x+\overline{\s})_{A_1}$. Therefore  $(x+\overline{\s})_{A_1}\supset (x+\overline{\s})_{A_2}$ and by symmetry, $(x+\overline{\s})_{A_1}=(x+\overline{\s})_{A_2}$. Proposition follows. $\square$
\end{proof}

 If $A$ and $B$ are apartments and $\psi:A\rightarrow B$ is an isomorphism, then $\psi$ induces a bijection still denoted $\psi$ between the sector-germs of $A$ and those of $B$. 

\begin{defn/prop}\label{defn/prop Addition}

Let $\s$ be a sector-germ. Let $x\in \I$ and $A_1$ be an apartment containing $x+\s$. Let $\mathrm{sgn}(\s)\in \{-,+\}$ be the sign of $\s$. Let $u\in \overline{C^v_f}$ and  $\psi_1:\A\rightarrow A_1$ be an isomorphism such that $\psi_1(\mathrm{sgn}(\s)\infty)=\s$. Then $\psi_1\big(\psi_1^{-1}(x+\mathrm{sgn}(\s)u)\big)$ does not depend on the choice of $A_1$ and of $\psi_1$ and we denote it $x+_\s u$. Moreover $x+_\s C^v_f=x+\s$ and $x+_\s \overline{C^v_f}=x+\overline{\s}$.

\end{defn/prop}

\begin{proof}
As the case where $\s$ is negative is similar, we assume that $\s$ is positive. 

We first prove the independence of the choice of isomorphism. Let $\psi_1':\A\rightarrow A_1$ be an isomorphism such that $\psi'_1( +\infty)=\s$. Then $\psi'^{-1}_1\circ\psi_1 \in W^v\ltimes Y$ 
fixes the direction $+\infty$ and thus $\psi'^{-1}_1\circ\psi_1$ is a translation of $\A$. Therefore \[\psi'^{-1}_1\circ \psi_1\big(\psi_1^{-1}(x)+u\big)=\psi'^{-1}_1\circ \psi_1\big(\psi_1^{-1}(x)\big)+u=\psi'^{-1}_1(x)+u,\] and thus \[\psi_1\big(\psi_1^{-1}(x+u)\big)=\psi'_1\big(\psi'^{-1}_1(x+u)\big).\]

Let now $A_2$ be an apartment containing $x+\s$ and $\psi_2:\A\rightarrow A_2$ be an isomorphism such that $\psi_2(+\infty)=\s$. From what has already been proved, we can assume that $\psi_2\circ \psi_1^{-1}=\phi$, where $\phi:A_1\overset{A_1\cap A_2}{\rightarrow} A_2$. 

As $x\in A_1\cap A_2$, $\phi(x)=x$ and thus \[\psi_1^{-1}(x)=\psi_2^{-1}(x).\]

Let $i\in \{1,2\}$. Then $\psi_i\big(\psi_i^{-1}(x)+C^v_f\big)$ is a sector with the base point $x$ and with the direction $\s$: $\psi_i\big(\psi^{-1}_i(x)+C^v_f\big)=x+\s$ (see Proposition~\ref{defnProp faces}). Moreover $\psi_i\big(\psi_i^{-1}(x)+\overline{C^v_f}\big)$ is the closure of $\psi_i\big(\psi_i^{-1}(x)+C^v_f\big)=x+\s$ in $A_i$ and thus $\psi_i\big(\psi_i^{-1}(x)+\overline{C^v_f}\big)=x+\overline{\s}$.

Consequently $\psi_1\big(\psi_1^{-1}(x+u)\big)\in x+\overline{\s}\subset A_1\cap A_2$. Thus \[\phi\bigg(\psi_1\big(\psi_1^{-1}(x+u)\big)\bigg)=\psi_1\big(\psi_1^{-1}(x+u)\big)=\psi_2\big(\psi_1^{-1}(x)+u\big)=\psi_2\big(\psi_2^{-1}(x)+u\big),\] which is our assertion. $\square$
\end{proof}

Through the end of this section, we fix a sector-germ $\s$. As the case when $\s$ is negative is similar to the case when it is positive, we assume that $\s$ is positive.

\begin{lemma}\label{lemCommutativite de ladition}
Let $x\in \I$ and $u,u'\in \overline{C^v_f}$. Then $(x+_\s u)+_\s u'=x+_\s(u+u')$.
\end{lemma}

\begin{proof}
Let $A$ be an apartment containing $x+\s$ and $\psi:\A\rightarrow A$ be such that $\psi(+\infty)=\s$.  One has $(x+_\s u)+_\s u'$, $x+_\s(u+u')\in A$. By definition, $\psi^{-1}(x+_\s u)=\psi^{-1}(x)+u$, thus \[(x+_\s u)+_\s u'=\psi(\psi^{-1}(x+_\s u)+u')=\psi(\psi^{-1}(x)+u+u')=x+_\s(u+u'),\] which proves the lemma. $\square$
\end{proof}

For $x,x'\in \I$, we set $U_\s(x,x')=\{(u,u')\in \overline{C^v_f}^2|\ x+_\s u=x'+_\s u'\}$.

\begin{lemma}\label{lemT(x,y) non vide}
Let $x,x'\in \I$. Then $U_\s(x,x')$ is nonempty.
\end{lemma}

\begin{proof}
Let $A$ be an apartment containing $\s$. Choose $a\in (x+\s)\cap A$ and $a'\in (x'+\s)\cap A$. Then $a+\s$ and $a'+\s$ are sectors of $A$ of the same direction and thus there exists $b\in (a+\s)\cap (a'+\s).$ By Definition/Proposition~\ref{defn/prop Addition}, there exist $u,u',v,v'\in C^v_f$ such that $a=x+_\s u$, $a'=x'+_\s u'$ and $b=a+_\s v=a'+_\s v'$. By Lemma~\ref{lemCommutativite de ladition}, $(u+v,u'+v')\in U_\s(x,x')$ and the lemma is proved.
\end{proof}

\subsection{Definition of distances of positive type and of negative type}\label{subDefinition_distances}

 Let $\Theta_+$ (resp. $\Theta_-$) be the set of pairs $(|\ .\ |,\s)$ such that $\s$ is a  positive (resp. negative) sector-germ and $|\ .\ |$ is a norm on $\A$.
 
 We now define the distance on $\I$ associated to $(|\ \cdot\ |,\s)$.  When $\I$ is a building the distance on $\I$ is usually defined as follows. One equips $\A$ with a euclidean $W^v$-invariant norm $|\ \cdot\ |$. For $x,y\in \A$, one chooses $g\in G$ such that $g.x,g.y\in \A$  and one sets $d(x,y)=|g.y-g.x|$. The fact $|\ .\ |$ is $W^v$ invariant implies that this distance is well defined. When $\I$ is a masure which is not a building, there exists no $W^v$-invariant norm on $\A$ (by Proposition~\ref{propNon existence d'une distance W-invariante}) and  there can exist pairs of points which are not contained in a common apartment. Thus we have to find another method to define a distance. For each pair of points, there exists a piecewise linear path joining them (where a piecewise linear path is a map $\gamma:[0,1]\rightarrow \I$ such that there exists $n\in \N$ and $t_0=0<t_1<\ldots<t_n=1$ such that for all $i\in \llbracket 0,n-1\rrbracket$, $\gamma|_{[t_i,t_{i+1}]}$ takes its values in some apartment $A_i$ and is an affine parametrization of the segment $[\gamma(t_i),\gamma(t_{i+1})]_{A_i}$). Thus we can try to define the distance between two points as the minimal length of a piecewise linear path joining them.  However to this end, we need to define a notion of a length for a path and the nonexistence of a $W^v$-invariant norm on $\A$ makes it difficult (because we cannot simply define it on  $\A$ and transport it to each
  apartment by using isomorphism of apartments). To avoid this problem, we  fix a sector $\s$, which enables us to define a length on every apartment containing $\s$. As the paths described above are difficult to handle --- for example we do not really understand the segments that are not increasing or decreasing for the Tits preorder --- we impose very strict conditions on the paths that we measure. In particular, we require $n=2$ (i.e at most one break-point in the paths) and  we require the directions of the segments to be contained in $\s$, which leads to the definition below.
 
\begin{defn/prop}\label{defnpropDistance}
Let $\theta=(|\ .\ |,\s)\in \Theta_+\cup \Theta_-$. Let $d_\theta:\I^2\rightarrow \R_+$ be defined by $d_ \theta(x,x')=\inf\{|u|+|u'|\ |\ (u,u')\in U_{\s}(x,x')\}$ for  $x,x'\in \I$. Then $d_\theta$ is a distance on $\I$.

\end{defn/prop}

\begin{proof}

By Lemma~\ref{lemT(x,y) non vide}, $d_\theta$ is well defined. Moreover it is clearly symmetric. 

Let us show the triangle inequality. Let $x,x',x''\in \I$. Let $\epsilon>0$ and let $(u,u')\in U_\s(x,x')$, $(v',v'')\in U_\s(x',x'')$ be such that $|u|+|u'|\leq d_\theta(x,x')+\epsilon$ and $|v'|+|v''|\leq d_\theta(x',x'')+\epsilon$. One has $x+_\s u=x'+_\s u'$ and $x'+_\s v'=x''+_\s v''$. Thus $x+_\s u+_\s v'=x'+_\s v'+_\s u' =x''+_\s v''+_\s u'$ (by Lemma~\ref{lemCommutativite de ladition}) and hence $(u+v',v''+u')\in U_\s(x,x'')$. Consequently, $d_\theta(x,x'')\leq |u+v'|+|v''+u'|\leq |u|+|v'|+|v''|+|u''|\leq d_\theta(x,x')+d_\theta(x',x'')+2\epsilon$, which proves the triangle inequality.

Let $x,x'\in \I$ be such that $d_\theta(x,x')=0$.  Then there exists $\big((u_n,u_n')\big)_{n\in \N}\in U_\s(x,x')^\N$ such that $u_n\rightarrow 0$ and $u'_n\rightarrow 0$. Let $n\in \N$. One has $x+\s\supset x+_\s u_n+ \s=x'+_\s u'_n+\s$ and thus $x+\s\supset \bigcup_{n\in \N}x'+u'_n+\s=x'+\s$. By symmetry, $x'+\s\supset x+\s$ and hence $x+\s=x'+\s$. Let $B$  be an apartment containing  $x$ and $\s$. By (MA2), $B\Supset \mathrm{cl}(x+\s)=\mathrm{cl}(x'+\s)$ and thus $x'\in B$. Therefore $x=x'$. $\square$
\end{proof}

 Thus we have constructed a distance $d_\theta$ for all $\theta\in \Theta_+\cup \Theta_-$. A distance of the form $d_{\theta_+}$ (resp. $d_{\theta_-}$) for some $\theta_+\in\Theta_+$ (resp. $\theta_-\in \Theta_-$) is called a \textit{distance of positive type} (resp. \textit{distance of negative type}). When $\I$ is a tree, we obtain a distance proportional to the distance of a Euclidean building on $\I$: there exists $k\in \R$ such that if $x,y\in \A=\R$, then  $d_\theta(x,y)=k|y-x|$ and the action of $G$ on $\I$ is isometric. 
 
 The choice of the norm has an influence on the metric on $\I$ but not on the topology defined on $\I$ (see Theorem~\ref{thmEquivalence des distances}). Independently of the choice of the norm  every pair of points is joined by a geodesic and there exists pairs of points joined by infinitely many geodesics (when $\dim \A\geq 2$, see Proposition~\ref{propI est geodesic}).

\subsubsection{Examples} We suppose that $\A$ is two-dimensional and  that $\bigcap_{i\in I} \ker \alpha_i=\{0\}$ (thus the root generating system defining $\A$ is associated with a size $2$ Kac-Moody matrix). We determine the restriction of $d_\theta$ to $\A$ when $\theta=(|\ \cdot\ |,+\infty)$ for different choices of a norm on $\A$.

Write $I=\{1,2\}$. Let $u_1,u_2\in \A$  be such that $\alpha_1(u_1)=1$, $\alpha_2(u_1)=0$, $\alpha_1(u_2)=0$ and $\alpha_2(u_2)=1$. Then $\overline{C^v_f}=\{x_1 u_1+x_2u_2|(x_1,x_2)\in (\R_+)^2\}$.

We begin by determining $d_\theta|_{\A^2}$, when $\theta$ is associated to any Euclidean norm. For $x=(x_1,x_2)\in\A$, one sets $|x|_2=\sqrt{x_1^2+x_2^2}$. Let $\theta_2=(|\ .\ |_2,+\infty)$. We now determine the restriction of $d_{\theta_2}$ to $\A$.

\begin{prop}\label{propExample_dtheta_restricted_apartment}
Let $x,y\in \A$. Then: \[d_{\theta_2}(x,y)=\left\{\begin{aligned} &|y-x|_2 &\mathrm{\ if\ } y-x\in \pm \overline{C^v_f}\\& |\alpha_1(y-x)u_1|_2+|\alpha_2(y-x)u_2|_2 &\mathrm{\ if\ }y-x\notin \pm \overline{C^v_f}.\end{aligned}\right.\] Moreover, if   $\alpha_1(x)\leq \alpha_1(y)$ and $\alpha_2(y)\leq \alpha_2(x)$, then  $(\alpha_1(y-x)u_1,\alpha_2(x-y)u_2)\in U_{+\infty}(x,y)$ and $d_{\theta_2}(x,y)=|\alpha_1(y-x)u_1|_2+|\alpha_2(x-y)u_2|_2$. 
\end{prop}

For $x\in \A$ and $u\in \A$, we denote by $x+\R_+u$ the set $\{x+tu|t\in \R_+\}$. 

\begin{lemma}\label{lemLocal_minimum_distance}
Let $x,y\in \A$. Let $f:\A\rightarrow \R_+$ be defined by $f(z)=|z-x|_2+|y-x|_2$, for $z\in \A$. Then the set of points at which $f$ admits a local minimum is the segment $[x,y]$ and $\min f=|y-x|_2=f(z)$, for all $z\in [x,y]$.
\end{lemma}

\begin{proof}
Using a translation and a rotation, we may assume that $x=0$ and $y=(y_1,0)$ for some $y_1\in \R$. The lemma follows by straightforward computations of $\frac{\partial f}{\partial z_2} (z_1,z_2)$ and $z_1\mapsto f(z_1,0)$. $\square$
\end{proof}

We now prove Proposition~\ref{propExample_dtheta_restricted_apartment}. Let $x,y\in \A$ and $f:\A\rightarrow \R_+$ be defined by $f(z)=|z-x|_2+|y-x|_2$, for $z\in \A$. Suppose $y-x\in \overline{C^v_f}$. Then $y=x+y-x$ and thus \[d_{\theta_2}(x,y)\leq |0|_2+|y-x|_2=|y-x|_2=\min f\leq d_{\theta_2}(x,y),\] which proves that $d_{\theta_2}(x,y)=|y-x|_2$. 

Suppose $y-x\notin \pm \overline{C^v_f}$. Then $[x,y]\cap (x+\overline{C^v_f})\cap (y+\overline{C^v_f})$ is empty and thus $\min \{f(z)|z\in (x+\overline{C^v_f})\cap (y+\overline{C^v_f})\}>|y-x|_2$. By Lemma~\ref{lemLocal_minimum_distance}, the minimum of $f$ on $(x+\overline{C^v_f})\cap (y+\overline{C^v_f})$ is attained on  the boundary $\partial\big((x+\overline{C^v_f})\cap (y+\overline{C^v_f})\big)$ of $(x+\overline{C^v_f})\cap (y+\overline{C^v_f})$. Suppose for example that $\alpha_1(x)\leq \alpha_1(y)$ and $ \alpha_2(y)\leq \alpha_2(x)$. Let $z\in \A$ be such that $\alpha_1(z)=\alpha_1(y)$ and $\alpha_2(z)=\alpha_2(x)$. Then $\partial\big((x+\overline{C^v_f})\cap (y+\overline{C^v_f})\big)=(z+\R_+ u_1)\cup (z+\R_+ u_2)$. Let $z'\in z+\R_+u_1$. Write $z'=z+tu_1$. Then $f(z')=|z'-x|_2+|z'-y|_2=|z-x|_2+t|u_1|_2+|z+tu_1-y|_2\geq f(z)$. By symmetry we deduce that $\min \{f(z')|\ z'\in (x+\overline{C^v_f})\cap (y+\overline{C^v_f})\}=f(z)=\alpha_1(y-x)|u_1|_2+\alpha_2(x-y)|u_2|_2$, and the proposition follows. $\square$

We now determine $d_\theta|_{\A^2}$, when $\theta$ is associated to a certain norm $1$ on $\A$.

\begin{prop}\label{propExample_dtheta_restricted_2}
Define $|\ \cdot\ |_1:\A\rightarrow \R_+$ by $|x|_1=|\alpha_1(x)|+|\alpha_2(x)|$ for $x\in \A$. Let $\theta_1=(|\ \cdot\ |_1,+\infty)$. Then $d_{\theta_1}(x,y)=|y-x|_1$ for all $x,y\in \A$.
\end{prop}

\begin{proof}
Let $x,y\in \A$.  Let $(v,v')\in U_{+\infty}(x,y)$. Then for $i\in \{1,2\}$, $\alpha_i(v)\geq \max \big(0,\alpha_i(y-x)\big)$ and $\alpha_i(v')\geq \max \big(0,\alpha_i(x-y)\big)$. Thus $|v|_1+|v'|_1\geq |y-x|_1$ and hence $d_{\theta_1}(x,y)\geq |y-x|_1$. 

Suppose $y-x\in \overline{C^v_f}$. Then $(y-x,0)\in U_{+\infty}(x,y)$, thus $d_{\theta_1}(x,y)\leq |y-x|_1$ and hence $d_{\theta_1}(x,y)=|y-x|_1$.

Suppose $y-x\notin \pm \overline{C^v_f}$. Suppose for example $\alpha_1(x)\geq \alpha_1(y)$ and $\alpha_2(x)\leq \alpha_2(y)$. Then $\big(\alpha_2(y-x)u_2,\alpha_1(x-y)u_1\big)\in U_{+\infty}(x,y)$ and thus $d_{\theta_1}(x,y)\leq |y-x|_1$. Therefore $d_{\theta_1}(x,y)=|y-x|_1$ and the proposition follows. $\square$
\end{proof}

\subsection{Study on the apartments containing $\s$}\label{subStudy_on_apartments_containing_s}

We now study the $d_\theta$, for $\theta\in \Theta_+\cup \Theta_-$. In order to simplify the notation and by symmetry, we will mainly take $\theta\in \Theta_+$.

Fix $\theta\in \Theta_+$. Write $\theta=(|\ .\ |,\s)$, where $|\ .\ |$ is a norm and $\s$ is a sector-germ. 
 
\begin{lemma}\label{lemDistances restricted to apartments containing s}
Let $A$ and $B$ be two apartments containing $\s$. Set $\rho:\I\overset{\s}{\rightarrow} A$ and $\phi:A\overset{A\cap B}{\rightarrow} B$. Then: \begin{enumerate}

\item\label{itDtheta norme} the distance $d_{\theta}|_{A^2}$ is induced by some norm on $A$,

\item\label{itCompatibility retraction addition} for all $x\in \I$ and $u\in \overline{C^v_f}$, $\rho(x+_\s u)=\rho(x)+_\s u$,

\item\label{itRetraction fixing s 1Lipschitz continuous} the retraction $\rho:(\I,d_\theta)\rightarrow (A,d_{\theta}|_{ A^2})$ is $1$-Lipschitz,

\item\label{itIso_fixing s continuous} the map $\phi:(A,d_{\theta}|_{A^2})\rightarrow (B,d_{\theta}|_{B^2})$ is an isometry.

\end{enumerate}
\end{lemma}

\begin{proof}
Let us prove~\ref{itDtheta norme}. Let $\psi:\A\rightarrow A$ be such that $\psi(+\infty)=\s$. 
Let $|\ .\ |':\A\rightarrow \R^+$ be defined by $|a|'=d_\theta\big(\psi(a),\psi(0)\big)$ for  $a\in \A$. 

For $a_1,a_2\in \A$, set $V(a_1,a_2)=\{(u_1,u_2)\in \overline{C^v_f}^2| a_1-a_2=u_2-u_1\}$. Let $(a_1,a_2)\in A$.
 Let $i\in \{1,2\}$ and $u_i\in \overline{C^v_f}$. Then $a_i+_\s u_i=\psi(\psi^{-1}(a_i)+u_i)$ and thus
 \[U_\s(a_1,a_2)=V\big(\psi^{-1}(a_1),\psi^{-1}(a_2)\big).\]

Let $a_1,a_2$. Then $U_\s(a_1,a_2)=V\big(\psi^{-1}(a_1),\psi^{-1}(a_2)\big)=V\big(\psi^{-1}(a_1)-\psi^{-1}(a_2),0\big)$. Consequently $d_\theta(a_1,a_2)=|\psi^{-1}(a_1)-\psi^{-1}(a_2)|'$. 
It remains to prove that $|\ .\ |'$ is a norm on $\A$. Let $a_1,a_2\in \A$. Then 
\[|a_1+a_2|'=d_\theta\big( \psi(a_1+a_2),\psi(0)\big)\leq d_\theta \big(\psi(a_1+a_2),\psi(a_1)\big)+d_\theta\big(\psi(a_1),\psi(0)\big)\] by
 Definition/Proposition~\ref{defnpropDistance}. As $V(a_1+a_2,a_1)=V(a_2,0)$, we deduce that
  $d_\theta\big(\psi(a_1+a_2),\psi(a_1)\big)=d_\theta \big(\psi(a_2),\psi(0)\big)$ and hence $|a_1+a_2|'\leq |a_1|'+|a_2|'$. 
  
  Let $t\in \R$ and $a\in \A$. As $V(0,ta)=tV(0,a)$, we deduce that $|ta|'=|t| |a|'$, which proves~\ref{itDtheta norme}.

Let us prove~\ref{itCompatibility retraction addition}. Let $x\in \I$ and $A_x$ be an apartment containing $x+\s$.  Let $\phi:A_x\overset{A_x\cap A}{\rightarrow} A$. Let $\psi_x:\A\rightarrow A_x$ be such that $\psi_{x}(+\infty)=\s$ and $\psi_A=\phi\circ \psi_x$. Then $\psi_{A}(+\infty)=\s$. Let $u\in \overline{C^v_f}$. Then  by Definition/Proposition~\ref{defn/prop Addition}, $A_x\ni x+_\s u$ and $A\ni \rho(x)+_\s u$. Therefore \[\rho(x+_\s u)=\phi(x+_\s u)=\phi\circ \psi_x(\psi^{-1}_x(x)+u)=\psi_A(\psi_x^{-1}(x)+u)\] and \[\rho(x)+_\s u=\psi_A\bigg(\psi_A^{-1}\big(\phi(x)\big)+u\bigg)=\psi_A\big(\psi_x^{-1}(x)+u \big)=\rho(x+_\s u),\] which proves~\ref{itCompatibility retraction addition}.

By~\ref{itCompatibility retraction addition}, for all $x,x'\in \I$, $U_\s\big(\rho(x),\rho(x')\big)\supset U_s(x,x')$, which proves~\ref{itRetraction fixing s 1Lipschitz continuous}. By~\ref{itRetraction fixing s 1Lipschitz continuous}, $\phi^{-1}:(B,d_{\theta}|_{B^2})\rightarrow (A,d_{\theta}|_{A^2})$ is $1$-Lipschitz. By symmetry, $\phi:(A,d_{\theta}|_{A^2})\rightarrow (B,d_{\theta}|_{B^2})$ is $1$-Lipschitz, which proves~\ref{itIso_fixing s continuous}.
$\square$

\end{proof}

\begin{lemma}\label{lemLipschitziannite de l'addition}
 Let $d'$ be a distance on $\A$ induced by some norm on $\A$. Define $d_{\theta,d'}:\I\times \overline{C^v_f}\rightarrow \R_+$ by $d_{\theta,d'}\big((x,u),(x',u')\big)=d_\theta(x,x')+d'(u,u')$ for  $(x,u),(x',u')\in \I\times\overline{C^v_f}$. Then the map $(\I\times \overline{C^v_f},d_{\theta,d'}) \rightarrow (\I,d_\theta)$ continuous defined by $(x,u)\mapsto x+_\s u$ is Lipschitz.
\end{lemma}

\begin{proof}
Using isomorphisms of apartments, we may assume that $\s$ is contained in $\A$. Since all norms on $\A$ are equivalent, it suffices to prove the assertion for a particular choice of $d'$. We choose $d'=d_{\theta}|_{\A^2}$, which is possible by Lemma~\ref{lemDistances restricted to apartments containing s}~(\ref{itDtheta norme}). We regard $\overline{C^v_f}$ as a subset of $\I$.

 Let $(x,u),(x',u')\in \I\times \overline{C^v_f}$. Let $\epsilon>0$. Let   $(u,u')\in U_\s(x,x')$ and $(v,v')\in U_\s(u,u')$ be such that $|u|+|u'|\leq d_\theta(x,x')+\epsilon$ and $|v|+|v'|\leq d_\theta(u,u')+\epsilon$. By Lemma~\ref{lemCommutativite de ladition}, $(u+v,u'+v')\in U_\s(x+_\s u,x'+_\s u')$, thus $d_\theta(x+u,x'+u')\leq |u|+|v|+|u'|+|v'|\leq d_\theta(x,x')+d_\theta(u,u')+2\epsilon$ and hence $d_\theta(x+u,x'+u')\leq d_\theta(x,x')+d_\theta(u,u')=d_{\theta,d'}\big((x,u),(x',u')\big)$. Lemma follows.  $\square$
\end{proof}

\begin{lemma}\label{rqueDistance atteinte} For all $x,x'\in \I$, there exists $(u,u')\in U_\s(x,x')$ such that  $d_\theta(x,x')=|u|+|u'|$.

\end{lemma}

\begin{proof}
Let $x,x'\in \I$ and let $\big((u_n,u'_n)\big)\in U_\s(x,x')^\N$ be such that $|u_n|+|u'_n|\rightarrow d_\theta(x,x')$. Then $(|u_n|),(|u'_n|)$  are bounded and thus after extraction we  can assume that $(u_n)$ and $(u'_n)$ converge in $(\overline{C^v_f},|\ .\ |)$. Lemma~\ref{lemLipschitziannite de l'addition} implies that $(\lim u_n,\lim u'_n)\in U_\s(x,x')$, which proves our assertion. $\square$

\end{proof}

\subsection{Geodesics in $\I$}

Fix $\theta=(|\ .\ |,\s)\in \Theta_+$. We now prove that for all $x_1,x_2\in \I$, there exists a geodesic for $d_\theta$ between $x_1$ and $x_2$. However we prove that when $\dim \A\geq 2$, such a geodesic is not unique (even if $\theta$ is associated with a Euclidean norm). The non-uniqueness  already appeared on the examples of Subsection~\ref{subDefinition_distances} (see Proposition~\ref{propExample_dtheta_restricted_apartment} and Proposition~\ref{propExample_dtheta_restricted_2}). As a comparison, Euclidean buildings equipped with their usual metrics are CAT(0)  and thus uniquely geodesic. 

  Using isomorphisms of apartments, we may assume that $\s=+\infty$. For all $x\in \A$ and $u\in \overline{C^v_f}$, $x+_{+\infty} u=x+u$. To simplify the notation we write $+$ instead of $+_{+\infty}$.

\begin{lemma}\label{lemGeodesiques}
\begin{enumerate} 

\item\label{itFormule 1 sur les distances} Let $x_1,x_2\in \I$ and let $(u_1,u_2)\in U_{+\infty}(x_1,x_2)$ be such that $d_\theta(x_1,x_2)=|u_1|+|u_2|$. Then for both $i\in \{1,2\}$ and all $t,t'\in [0,1]$, \[d_\theta(x_i+ tu_i,x_i+ t'u_i)=|t'-t||u_i|\] and \[d_\theta(x_1+ tu_1,x_2+ t'u_2)=(1-t)|u_1|+(1-t')|u_2|.\]

\item\label{itFormule 2 sur les distances} Let $x\in \A$ and $(u_1,u_2)\in U_{+\infty}(0,x)$ be such that $d_\theta(0,x)=|u_1|+|u_2|$. Then for all $t_1,t_1',t_2,t_2'\in [0,1]$ such that $t_1\leq t_1'$ and $t_2\leq t_2'$, \[d_\theta(t_1u_1-t_2u_2,t_1'u_1-t_2' u_2)=(t_1'-t_1)|u_1|+(t_2'-t_2)|u_2|.\]
\end{enumerate}
\end{lemma}

\begin{proof} Let $t,t'\in [0,1]$. We assume $t\leq t'$. Let $i\in \{1,2\}$ and let $j$ be such that $\{i,j\}=\{1,2\}$. As $x_i+ u_i=x_j+ u_j$, \[\begin{aligned} d_\theta(x_1,x_2)\leq d_\theta(x_i,x_i+tu_i)+& d_\theta(x_i+ tu_i,x_i+ t'u_i)\\& + d_\theta(x_i+ t'u_i,x_i+ u_i) +d_\theta(x_j+ u_j,x_j).\end{aligned}\]

By the definition of $d_\theta$, $d_\theta(x_i,x_i+ tu_i)\leq t|u_i|$, $d_\theta(x_i+ tu_i,x_i+ t'u_i)\leq (t'-t)|u_i|$, 
$d_\theta(x_i+ t'u_i, x_i+ u_i)\leq (1-t')|u_i|$ and $d_\theta(x_j+ u_j,x_j)\leq |u_j|$.

As $ d_\theta(x_1,x_2)=|u_1|+|u_2|=t|u_i|+(t'-t)|u_i|+(1-t')|u_i|+|u_j|,$ we deduce that
 $d_\theta(x_i,x_i+ tu_i)= t|u_i|$, $d_\theta(x_i+ tu_i,x_i+ t'u_i)= (t'-t)|u_i|$,  
 $d_\theta(x_i+ t'u_i, x_i+ u_i)= (1-t')|u_i|$ and $d_\theta(x_j+ u_j,x_j)= |u_j|$.

We no longer assume $t\leq t'$. One has \[\begin{aligned} d_\theta(x_1+ tu_1,x_2+ t'u_2) & \geq d_\theta(x_1,x_2)-d_\theta(x_1,x_1+ tu_1)-d_\theta(x_2,x_2+ t'u_2)\\ &=(1-t)|u_1|+(1-t')|u_2|.\end{aligned}\]

Moreover \[  \begin{aligned} d_\theta(x_1+ tu_1,x_2+ t'u_2) & \leq d_\theta(x_1+ tu_1,x_1+ u_1)+d_\theta(x_2+ u_2,x_2+ t'u_2)\\ & =(1-t)|u_1|+(1-t')|u_2|,\end{aligned}\] which proves~\ref{itFormule 1 sur les distances}. A similar argument proves~\ref{itFormule 2 sur les distances}. $\square$

\end{proof}

\begin{remark}
When we are in the situation of Proposition~\ref{propExample_dtheta_restricted_apartment}, then the $u_1$ and $u_2$ of Lemma~\ref{lemGeodesiques} are in the boundary of $\overline{C^v_f}$.

\end{remark}

\begin{prop}\label{propI est geodesic}
Equip $\I$ with $d_\theta$. For all $x_1,x_2\in \I$, there exists a geodesic from $x_1$ to $x_2$. Moreover, if $\dim \A \geq 2$, there exists a pair $(x_1,x_2)\in \I^2$ such that there are infinitely many geodesics from $x_1$ to $x_2$.
\end{prop}

\begin{proof}
Let $x_1,x_2\in \I$ be such that $x_1\neq x_2$. Let $(u_1,u_2)\in U_{+\infty}(x_1,x_2)$ be such that $|u_1|+|u_2|=d_\theta(x_1,x_2)$. Let $a_1=\frac{|u_1|}{|u_1|+|u_2|}$ and $a_2=1-a_1$. Set $\frac{1}{0}u_1=\frac{1}{0}u_2=0$. 

Let $\gamma:[0,1]\rightarrow \I$ be defined by  $\gamma(t)=x_1+ \frac{t}{a_1}u_1$ if $t\in [0,a_1]$ and $\gamma(a_1+t)=x_2+ (1-\frac{t}{a_2})u_2$ if $t\in [0,a_2]$.   Then by Lemma~\ref{lemGeodesiques}~(\ref{itFormule 1 sur les distances}), for all $t,t'\in [0,1]$, $d_\theta(\gamma(t),\gamma(t'))=|t'-t|(|u_1|+|u_2|)$ and hence $\gamma$ is a geodesic from $x_1$ to $x_2$.

Let now  $x\in \A\backslash (\overline{C^v_f}\cup\overline{-C^v_f})$. Let us   construct infinitely many geodesics joining $0$ to $x$.  Let $(u_1,u_2)\in U_{+\infty} (0,x)$ be such that $d_\theta(0,x)=|u_1|+|u_2|$.  One has $x=u_1-u_2$ and thus $u_1,u_2\neq 0$. By Lemma~\ref{lemGeodesiques}, for $z\in [0,1]$, one has  $d_\theta(0,zu_1)+d_\theta(zu_1,x)=d_\theta(0,x)$. Thus our idea is to concatenate a geodesic from $0$ to $zu_1$ and a geodesic from $zu_1$ to $x$.   

Let $z\in [0,1]$. Set $t_z=\frac{z|u_1|}{|u_1|+|u_2|}$. Let $\gamma_z:[0,1]\rightarrow \A$ be defined by: \[\gamma_z(t)=\left\{\begin{aligned}& t\frac{|u_1|+|u_2|}{|u_1|}u_1 &\mathrm{\  for\ }t\in [0,t_z]\\& zu_1+\frac{|u_1|+|u_2|}{(1-z)|u_1|+|u_2|}(t-t_z)\big((1-z)u_1-u_2\big)& \mathrm{\ for\ }t\in [t_z,1].\end{aligned}\right.\]

Let $t,t'\in [0,1]$. First assume $0\leq t\leq t'\leq t_z$. Then by Lemma~\ref{lemGeodesiques}~(\ref{itFormule 2 sur les distances}), \[d\big(\gamma_z(t),\gamma_z(t')\big)=d(t\frac{|u_1|+|u_2|}{|u_1|}u_1,t'\frac{|u_1|+|u_2|}{|u_1|}u_1)=(t'-t)(|u_1|+|u_2 |).\]

Assume $t_z\leq t\leq t'\leq 1$. Let $\tilde{t}=\frac{|u_1|+|u_2|}{(1-z)|u_1|+|u_2|} (t-t_z)$ and $\tilde{t}'=\frac{|u_1|+|u_2|}{(1-z)|u_1|+|u_2|} (t'-t_z)$.

 Then by Lemma~\ref{lemGeodesiques}~(\ref{itFormule 2 sur les distances}), \[\begin{aligned} d\big(\gamma_z(t),\gamma_z(t')\big)&= d\big((\tilde{t}(1-z)+z)u_1-\tilde{t} u_2 ,(\tilde{t}'(1-z)+z)u_1-\tilde{t}' u_2\big)\\ & =(\tilde{t}'-\tilde{t})((1-z)|u_1|+|u_2|)\\ &= (t'-t)(|u_1|+|u_2|) .\end{aligned}\]
 
 Assume  $t\leq t_z\leq t'$. Then by Lemma~\ref{lemGeodesiques}~(\ref{itFormule 2 sur les distances}),\[\begin{aligned} d\big(\gamma_z(t),\gamma_z(t')\big)&= d\big(t\frac{|u_1|+|u_2|}{|u_1|}u_1 ,(\tilde{t}'(1-z)+z)u_1-\tilde{t}' u_2)\\ & =\big(\tilde{t}'(1-z)+z-t\frac{|u_1|+|u_2|}{|u_1|}\big)|u_1|+\tilde{t}'|u_2|\\ &= (t'-t)(|u_1|+|u_2|) .\end{aligned}\]
 
 Therefore, $\gamma_z$ is a geodesic from $0$ to $x$. Moreover, as $x\notin \overline{C^v_f}\cup \overline{-C^v_f}$, $\R u_1 \neq \R u_2$, thus  $\gamma_z([0,1])\neq \gamma_{z'}([0,1])$ for all $z\neq z'$ and the proposition is proved. $\square$

\end{proof}

\subsection{Equivalence of the distances of positive type}

The aim of this section is to show that if $\theta_1,\theta_2\in \Theta_+$, then $d_{\theta_1}$ and $d_{\theta_2}$ are equivalent. Fix a norm $|\ .\ |$ on $\A$. 

Fix two adjacent positive sector-germs $\s$ and $\s'$ and set $\theta=(|\ .\ |,\s)$ and $\theta'=(|\ .\ |,\s)$. We begin by proving the existence of $\ell\in \R_+$ such that $d_{\theta'}\leq \ell d_{\theta}$ (see Lemma~\ref{lemEquivalence des distances}).

 Fix an apartment $A_0$ containing $\s$ and $\s'$, which exists by (MA3). Let $\rho_\s:\I\overset{\s}{\rightarrow} A_0$ and $\rho_{s'}:\I\overset{\s'}{\rightarrow} A_0$. 


\begin{lemma}\label{lemegalite des distances pour les apparts contenant q et q'}

There exists $\ell_0\in \R_{>0}$ such that for every apartment  $B$ containing $\s$ and $\s'$,  for all $x,x'\in B$, $d_{\theta'}(x,x')\leq \ell_0 d_{\theta}(x,x')$.
\end{lemma}

\begin{proof}
By Lemma~\ref{lemDistances restricted to apartments containing s}~(\ref{itDtheta norme}) and the fact that all the norms on $\A$ are equivalent, there exists $\ell_0\in \R_{>0}$ such that for all $x,x'\in A$, $d_{\theta'}(x,x')\leq \ell_0  d_{\theta}(x,x')$. Let $B$ be an apartment containing $\s$ and $\s'$.  Let $x,x'\in B$. By Lemma~\ref{lemDistances restricted to apartments containing s}~(\ref{itIso_fixing s continuous}), $d_\theta(x,x')=d_{\theta}(\rho_\s(x),\rho_\s(x'))$ and  $d_{\theta'}(\rho_{\s'}(x),\rho_{\s'}(x'))=d_{\theta'}(x,x')$. Moreover $\rho_{\s'|B}=\rho_{\s|B}$, which proves the lemma. $\square$
\end{proof}

We now fix an apartment $B_0$ containing $\s$ but not $\s'$. Let $\F_\infty$ be the sector-panel direction dominated by $\s$ and $\s'$.   Using Lemma~\ref{lemDecoupage d'un appartement en 2}, one  writes $B_0=D_1\cup D_2$, where  $D_1$ and $D_2$ are two opposite half-apartments whose wall contains $\F_\infty$ and such that $D_i\cup \s$ is contained in some apartment $B_i$ for both $i\in \{1,2\}$. We assume that $D_1\supset \s$. 

Let $M_0$ be a wall of $A_0$ containing $\F_\infty$ and $t_0:A_0\rightarrow A_0$ be the reflection with respect to $M_0$. 
\begin{lemma}\label{lemImage des differentes retractions}
One has: $\left\{\begin{aligned} & \rho_{\s}(x) =\rho_{\s'}(x) &\mathrm{\ if\ }x\in D_1\\ & \rho_{\s}(x)=\tilde{t}\circ \rho_{\s'}(x) &\mathrm{\ if\ }x\in D_2\end{aligned}\right.,$ where $\tilde{t}=\tau\circ t_0$, for some translation $\tau$ of $A_0$.
\end{lemma}

\begin{proof}
Let $\rho_{\s,B_1}:\I\overset{\s}{\rightarrow}B_1$ and $\rho_{\s',B_1}:\I\overset{\s'}{\rightarrow}B_1$.

Let $\phi_i:B_0\overset{B_0\cap B_i}{\rightarrow}B_i$, for $i\in\{1,2\}$ and $\phi:B_2\overset{B_1\cap B_2}{\rightarrow} B_1$. Let $t$ be the reflection of $B_1$ with respect to  $D_1\cap D_2$. By Lemma~\ref{lem3 appartements s'intersectant bien}, the following  diagram is commutative: \[\xymatrix{ B_0\ar[d]^{\phi_1}\ar[r]^{\phi_2} & B_2\ar[d]^{\phi}\\ B_1\ar[r]^{t}&B_1.}\]
Let $x\in D_1$. Then $\rho_{\s,B_1}(x)=x=\rho_{\s',B_1}(x)$. Let $\phi_3:B_1\overset{B_1\cap A_0}{\rightarrow} A_0$. Then $\rho_{\s}(x)=\phi_3 (\rho_{\s,B_1}(x))=\phi_3 (\rho_{\s',B_1}(x))=\rho_{\s'}(x)$.

  Let $x\in D_2$. One has $\rho_{\s,B_1}(x)=\phi_1(x)$ and $\rho_{\s',B_1}(x)=\phi(x)$ and thus $\rho_{\s,B_1}(x)=t\circ \rho_{\s',B_1}(x)$.  Let $\tilde{t}$  be such that the following diagram commutes: \[\xymatrix{B_1 \ar[d]^{\phi_3}\ar[r]^{t} & B_1\ar[d]^{\phi_3}\\ A_0\ar[r]^{\tilde{t}}&A_0.}\]
  
  Then $\rho_{\s}(x)=\tilde{t}\circ \rho_{\s'}(x)$.
  
   Moreover $\tilde{t}$ fixes  $\phi_3(D_1\cap D_2)$, which contains $\F_\infty$. Thus  $\tilde{t}=\tau\circ t_0$ for some translation $\tau$ of $A_0$ (by Lemma~\ref{lemAutomorphisme d'appart fixant un mur}). $\square$

\end{proof}

By Lemma~\ref{lemDistances restricted to apartments containing s}~(\ref{itDtheta norme}) and since every affine map on $A_0$ is Lipschitz, there exists $\ell_1\in \R_+$ such that $t_0:(A_0,d_{\theta'})\rightarrow (A_0,d_{\theta'})$ is $\ell_1$-Lipschitz. As $t_0$ is an involution, $\ell_1\geq 1$.

\begin{lemma}\label{lemComparaison des distances sur les demi-apparts}
Let $\ell_0$ be as in Lemma~\ref{lemegalite des distances pour les apparts contenant q et q'}. Then for all $x,x'\in B_0$, $d_{\theta'}(x,x')\leq \ell_0 \ell_1 d_{\theta}(x,x')$.
\end{lemma}

\begin{proof}
 
 Let $i\in \{1,2\}$ and $x,x'\in D_i$. By Lemma~\ref{lemDistances restricted to apartments containing s}~(\ref{itIso_fixing s continuous}),  $d_\theta(x,y)=d_\theta(\rho_\s(x),\rho_\s(x'))$ and $d_{\theta'}(x,x')=d_{\theta'}(\rho_{\s'}(x),\rho_{\s'}(x'))$.  By Lemma~\ref{lemImage des differentes retractions}, for all $x,x'\in D_i$, $d_{\theta'}(x,x')\leq \ell_0\ell_1 d_\theta(x,x')$. 
 
Let $x,x'\in B_0$. Assume that $x\in D_1$ and $x'\in D_2$. Let $m\in [x,x']_{B_0}\cap D_1\cap D_2$. Then $d_{\theta'}(x,x')\leq d_{\theta'}(x,m)+d_{\theta'}(m,x')\leq \ell_0\ell_1 \big(d_\theta(x,m)+d_\theta(m,x')\big)$. By Lemma~\ref{lemDistances restricted to apartments containing s}~(\ref{itDtheta norme}), $d_\theta(x,m)+d_\theta(m,x')=d_\theta(x,x')$ and the lemma follows.
 $\square$
\end{proof}

\begin{lemma}\label{lemCaracterisation des fonctions lipschitziennes}
Let $(X,d_X)$ be a metric space, $f:(\I,d_{\theta})\rightarrow (X,d_X)$ be a map and $k\in \R_+$. Then $f$ is $k$-Lipschitz if and only if for every apartment $A$ containing $\s$, $f|_{A}$ is $k$-Lipschitz.
\end{lemma}

\begin{proof}
One implication is clear.  Assume that for  every apartment $A$ containing $\s$, $f|_{A}$ is $k$-Lipschitz. Let $x,x'\in \I$ and $A_{x},A_{x'}$ be apartments containing $x+\s$ and $x'+\s$. Let $(u,u')\in U_\s(x,x')$ be such that $|u|+|u'|=d_\theta(x,x')$, which exists by Lemma~\ref{rqueDistance atteinte}. Then $x+_\s u\in A_x$ and $x'+_\s u'\in A_{x'}$. One has \[d_X\big(f(x),f(x')\big)\leq d_X\big(f(x),f(x+_\s u)\big)+d_X\big(f(x'+_\s u'),f(x')\big)\leq k(|u|+|u'|)\leq kd_\theta(x,x').\] $\square$
\end{proof}

\begin{lemma}\label{lemEquivalence des distances}
One has $d_{\theta'}\leq \ell_0 \ell_1 d_{\theta}$.
\end{lemma}

\begin{proof}
By Lemma~\ref{lemCaracterisation des fonctions lipschitziennes}, Lemma~\ref{lemegalite des distances pour les apparts contenant q et q'} and Lemma~\ref{lemComparaison des distances sur les demi-apparts}, $\Id:(\I,d_{\theta})\rightarrow (\I,d_{\theta'})$ is $\ell_0 \ell_1$-Lipschitz, which proves the lemma.  $\square$
\end{proof}

\begin{theorem}\label{thmEquivalence des distances}
Let $\theta_1,\theta_2\in \Theta_+$. Then the metrics $d_{\theta_1}$ and $d_{\theta_2}$ are equivalent. 
\end{theorem}
\begin{proof}
For $i\in \{1,2\}$, write $\theta_i=(|\ .\ |_i,\s_i)$. As all the norms on $\A$ are equivalent, we may assume $|\ .\ |_1=|\ .\ |_2=|\ .\ |$. Let $\mathfrak{t}^0=\s_1,\ldots,\mathfrak{t}^n=\s_2$ be a gallery between $\s_1$ and $\s_2$. For  $i\in \llbracket 0,n\rrbracket$ set $\theta^i=(|\ .\ |,\mathfrak{t}^i)$. By an induction using Lemma~\ref{lemEquivalence des distances}, there exists $\ell\in \R_{>0}$ such that $d_{\theta_1}=d_{\theta^0}\leq \ell d_{\theta^n}=K d_{\theta_2}$. Theorem follows by symmetry. $\square$
\end{proof}

We thus obtain (at most) two topologies on $\I$: the topology $\mathscr{T}_+$ induced by $d_{\theta_+}$, for each $\theta_+\in \Theta_+$ and the topology $\mathscr{T}_-$ induced by $d_{\theta_-}$, for each $\theta_-\in \Theta_-$. We will see that when $\I$ is not a building, these topologies are different (see Corollary~\ref{corNon egalite de T+ et T-}).

\begin{cor}\label{corTopologie induite sur les appartements}
Let $A$ be an apartment of $\I$. Then the topology on $A$ induced by $\mathscr{T}_+$ is the affine topology on $A$.
\end{cor}

\begin{proof}
By Theorem~\ref{thmEquivalence des distances}, this topology is induced by $d_{(|\ .\ |,\mathfrak{t})}$ for  some positive sector-germ $\mathfrak{t}$ of $A$. Then Lemma~\ref{lemDistances restricted to apartments containing s}~(\ref{itDtheta norme}) concludes the proof. $\square$
\end{proof}

\begin{cor}\label{corRetractions lipschitziennes}
 Let $\rho$ be a retraction centered at a positive sector-germ, $A=\rho(\I)$, $B$ be an apartment and  $d_A$ (resp. $d_B$) be a distance on $A$ (resp. $B$) induced by a norm. Then: \begin{enumerate}

\item\label{itRetraction globally lipschitz} for   all $\theta\in \Theta_+$, $\rho:(\I,d_\theta)\rightarrow (A,d_A)$ is Lipschitz,

\item\label{itRetraction restricted}  the map $\rho_{|B}:(B,d_B)\rightarrow (A,d_A)$ is Lipschitz.
\end{enumerate} 
\end{cor}

\begin{proof}
By Theorem~\ref{thmEquivalence des distances} we may assume $\theta=(|\ .\ |,\mathfrak{t})$, where $\mathfrak{t}$ is the center of $\rho$. Then by Lemma~\ref{lemDistances restricted to apartments containing s}~(\ref{itRetraction fixing s 1Lipschitz continuous}), $\rho:(\I,d_\theta)\rightarrow (A,d_\theta)$ is Lipschitz  and  Lemma~\ref{lemDistances restricted to apartments containing s}~(\ref{itDtheta norme}) completes the proof. $\square$
\end{proof}

\begin{cor}\label{corIntersection de deux apparts}
Let $A,B$ be two apartments of $\I$. Then $A\cap B$ is a closed subset of $A$ (seen as an affine space).
\end{cor}

\begin{proof}
By Lemma~\ref{lemApparts fermes quand une retraction est continue}, $A$ and $B$ are closed for $\mathscr{T}_+$ and thus $A\cap B$ is closed for $\mathscr{T}_+$. Consequently it is closed for the topology induced by $\mathscr{T}_+$ on $A$, and Corollary~\ref{corTopologie induite sur les appartements} completes the proof. $\square$
\end{proof}

\begin{remark}\label{rqueConsequence geometrique de la non completude}
Suppose that $\I$ is not a building. Then by Proposition~\ref{propMasure non complete dans le cas denombrable}, for all $\theta_+\in \Theta_+$,  $(\I,d_{\theta_+})$ is incomplete.

 Let $\mathfrak{s}''$ be a positive sector-germ of $\I$, $\theta_+=(|\ .\ |,\mathfrak{s}'')$ and $(S_n)$ be an increasing sequence of sectors with the germ $\s''$. One says that $(S_n)$ is converging if there exists a retraction onto an apartment $\rho:\I\overset{\s''}{\rightarrow} \rho(\I)$ such that $(\rho(x_n))$ converges, where $x_n$ is the base point of $S_n$ for all $n\in \N$ and we call \textit{limit} of $(S_n)$ the set $\bigcup_{n\in\N}S_n$. One can show that the incompleteness of $(\I,d_\theta)$ implies the existence of a converging sequence of the direction $\s''$ whose limit is not a sector of $\I$. To prove this one can associate to each Cauchy sequence $(x_n)$ a sequence $(x_n')$ such that $d_{\theta_+}(x_n',x_n)\rightarrow 0$ and such that $x_n'+\s''\subset x'_{n+1}+\s''$ for all $n\in \N$. Then we show that $(x'_n)$ converges in $(\I,d_\theta)$ if, and only if the limit of $(x'_n+\s'')$ is a sector of $\I$.
\end{remark}

\subsection{Study of the action of $G$}

In this subsection, we show that  for every $g\in G$, the induced map $g:\I\rightarrow \I$ is Lipschitz for the distances of positive type.

\begin{lemma}\label{lemAction of G on addition}
Let $g\in G$ and $\s$ be a sector-germ of $\I$. Then for every $x\in \I$ and $u\in \overline{C^v_f}$, $g.(x+_\s u)=g.x+_{g.\s} u$.
\end{lemma}

\begin{proof}
Let $x\in \I$ and $u\in \overline{C^v_f}$. Let $A$ be an apartment containing $x+\s$. Let $A'=g.A$. Then $A'$ contains $\s'=g.\s$. Let $\psi:\A\rightarrow A$ be an isomorphism such that $\psi(+\infty)=\s$. Let $f:A\rightarrow A'$ be the isomorphism induced by $g$. Set $\psi'=f\circ \psi$. Then $\psi'(+\infty)=\s'$. 

As $x+_\s u\in A$, \[g.(x+_\s u)=f(x+_\s u)=f\circ\psi\big(\psi^{-1}(x)+u\big)=\psi'\big(\psi^{-1}\circ f^{-1}\big(f(x)\big)+u\big)=g.x+_{g.\s} u.\] $\square$
\end{proof}

\begin{theorem}\label{thmAction de G lipschitzienne}
Let $g\in G$ and $\theta\in \Theta_+$. Then $g:(\I,d_\theta)\rightarrow (\I,d_\theta)$ is Lipschitz.
\end{theorem}

\begin{proof}
Write $\theta=(|\ .\ |,\s)$. Let $\theta'=(|\ .\ |,g.\s)$. By Theorem~\ref{thmEquivalence des distances}, it suffices to prove that $g:(\I,d_\theta)\rightarrow (\I,d_{\theta'})$ is Lipschitz.

Let $x,x'\in \I$. By Lemma~\ref{lemAction of G on addition}, $U_{g.\s}(g.x,g.x')\supset U_\s(x,x')$, thus $d_{\theta'}(g.x,g.x')\leq d_\theta(x,x')$, which proves the theorem. $\square$
\end{proof}

\subsection{Case of a building}
In this subsection we  assume that $\I$ is a building. We show that the distances of positive type are equivalent to the usual distance.

Let $d_\A$ be a distance on $\A$ induced by some $W^v$-invariant euclidean norm $|\ .\  |$ on $\A$. Let $x,x'\in \I$, $A$ be an apartment containing $x,x'$ and $f:A\rightarrow \A$ be an isomorphism of apartments. One sets $d_\I(x,x')=d_\A\big(f(x),f(x')\big)$. Then $d_\I:\I\rightarrow \R_+$ is well defined and is a distance on $\I$ (see \cite{brown1989buildings} VI.3 for example). 
Recall that $\rho_{+\infty}:\I\overset{+\infty}{\rightarrow}\A$.

\begin{prop}\label{propEquivalence des distances pour un immeuble}
Let $\theta\in \Theta_+$. Then $d_\I$ and $d_\theta$ are equivalent.
\end{prop}

\begin{proof}
By Theorem~\ref{thmEquivalence des distances}, one can assume that $\theta=(|\ .\  |,+\infty)$. Let $k,\ell\in \mathbb{R}_{>0}$ be such that $kd_{\I}|_{\A^2}\leq d_{\theta}|_{\A^2}\leq \ell d_{\I}|_{\A^2}$, which exists by Lemma~\ref{lemDistances restricted to apartments containing s}~(\ref{itDtheta norme}). Let us first show that $\Id:(\I,d_\theta)\rightarrow (\I,d_\I)$ is $\frac{1}{k}$-Lipschitz.

 Let $A$ be an apartment containing $+\infty$. Let $x,x'\in A$. Then by Lemma~\ref{lemDistances restricted to apartments containing s}~(\ref{itIso_fixing s continuous}) and the fact that the restriction of $\rho_{+\infty}$ to $A$ is an isometry for $d_\I$, $d_\theta(x,x')=d_\theta(\rho_{+\infty}(x),\rho_{+\infty}(x'))\geq  k d_\I(\rho_{+\infty}(x),\rho_{+\infty}(x'))=k d_\I(x,x')$. From  Lemma~\ref{lemCaracterisation des fonctions lipschitziennes} we deduce that $\Id:(\I,d_\theta)\rightarrow (\I,d_\I)$ is $\frac{1}{k}$-Lipschitz.

Let $x,x'\in \I$.  By Corollary~\ref{corSegments germes de quartiers} there exist $n\in \N_{>0}$ and $x_0=x,x_1,\ldots,x_n=x'\in [x,x']$ such that $[x,x']=\bigcup_{i=0}^{n-1} [x_i,x_{i+1}]$ and such that $[x_i,x_{i+1}]\cup {+\infty}$ is contained in an apartment for all $i\in \llbracket 0,n-1\rrbracket$.  By Lemma~\ref{lemDistances restricted to apartments containing s}~(\ref{itIso_fixing s continuous}),
\[\begin{aligned}d_\theta(x,x')\leq \sum_{i=0}^{n-1}d_\theta(x_i,x_{i+1})&=\sum_{i=0}^{n-1}d_\theta(\rho_{+\infty}(x_i),\rho_{+\infty}(x_{i+1}))\\ &\leq \ell\sum_{i=0}^{n-1}d_\I(\rho_{+\infty}(x_i),\rho_{+\infty}(x_{i+1})) \\  &=\ell\sum_{i=0}^{n-1}d_\I(x_i,x_{i+1})=\ell d_\I(x,x'),\end{aligned}\]
which proves the proposition. $\square$
\end{proof}

\section{Mixed distances}\label{secDistances combinees}
In this section, we begin by proving that unless $\I$ is a building,  if  $\s_-$ is a negative sector-germ, then every retraction centered at $\s_-$ is not continuous for $\mathscr{T}_+$ see Subsection~\ref{subsecNon egalite de T+ et T-}).  To prove this we show that the set of vertices $\I_0$ contains no isolated points when $\I$ is not a building and then we apply finiteness results of \cite{hebertGK}. 

 This implies that $\mathscr{T}_+\neq \mathscr{T}_-$ and motivates the introduction of mixed distances, each of which is the sum of a distance of positive type with a distance of negative type. We then study them.

\subsection{Comparison of positive and negative topologies}\label{subsecNon egalite de T+ et T-}

Fix a norm $|\ .\ |$ on $\A$.

\begin{prop}\label{propNon discretion de I0}
Let  $\theta\in \Theta$. Then $\I_0$ is discrete in $(\I,d_\theta)$ if and only if $\I$ is a building.
\end{prop}

\begin{proof}
 Assume that $\I$ is a building. By Proposition~\ref{propEquivalence des distances pour un immeuble}, we can replace $d_\theta$ by the usual distance $d_\I$ on $\I$. By Lemma~\ref{lemIntersection de I0 et A}, $\I_0\cap \A=Y$, which is a lattice in $\A$. Let $\eta>0$ be such that for all $\lambda,\lambda'\in Y$, $d_\I(\lambda,\lambda') < \eta$ implies $\lambda=\lambda'$. Let $\lambda,\lambda'\in \I_0$ be such that $d_\I(\lambda,\lambda')<\eta$. Let $A$ be an apartment of $\I$ containing $\lambda$ and $\lambda'$ and $g\in G$ be such that $g.A=\A$. Then $d_\I(g.\lambda,g.\lambda')=d_\I(\lambda,\lambda')<\eta$, thus $\lambda=\lambda'$ and hence $\I_0$ is discrete in $\I$.

 Assume now that $\I$ is not a building and, thus, $W^v$ is infinite. By Theorem~\ref{thmEquivalence des distances}, we can assume that $\theta=(|\ .\ |,+\infty)$. Let $\epsilon >0$. Let us show that there exists $\lambda\in \I_0$ such that $d_\theta(\lambda,0)<2\epsilon$ and $\lambda\neq 0$. Let $M_0$ be a  wall of $\A$ containing $0$ such that for all consecutive walls $M_1$ and $M_2$ of the direction $M_0$, $d_\theta(M_1,M_2)<\epsilon$ (such a direction exists by Proposition~\ref{propNon existence d'une distance W-invariante}~(\ref{itemMurs denses})). Let $M$ be a wall such that $d_\theta(0,M)<\epsilon$ and such that $0\notin D$, where $D$ is the half-apartment of $\A$ delimited by $M$ and containing $+\infty$. By  Lemma~\ref{lemApparts branchant partout}, there exists an apartment $A$ such that $A\cap \A=D$.  Let $\phi:\A\overset{A\cap \A}{\rightarrow} A$ and $\mu=\phi(0)$. Let $x\in M$ be such that $d_\theta(0,x)<\epsilon$. Then by Lemma~\ref{lemDistances restricted to apartments containing s}~(\ref{itIso_fixing s continuous}), $d_\theta(\lambda,x)=d_\theta(0,x)$ and thus $d(\lambda,0)<2\epsilon$. As $\lambda\notin \A$, $\lambda\neq 0$ and we get the proposition. $\square$
\end{proof}

\begin{remark}\label{rquePoints limites de I0}
In fact, by Theorem~\ref{thmAction de G lipschitzienne} and since $G$ acts transitively on $\I_0$, we  proved that when $\I$ is not a building,  every point of $\I_0$ is a limit point.
\end{remark}

If $B$ is an apartment and $(x_n)\in B^\N$, one says that $(x_n)$ diverges to $\infty$, if for some isomorphism $f:B\rightarrow \A$, $|f(x_n)|\rightarrow +\infty$.

\begin{prop}\label{propNon egalite de T+ et T-}
 Assume that $\I$ is not a building. Let  $\s_{-}$ be a negative sector-germ of $\I$ and $\theta\in \Theta_+$. Equip $\I$ with $d_\theta$. Let $\rho_{-}$ be a retraction centered at $\s_{-}$ and $(\lambda_n)\in \I_0^\N$ be an injective and converging sequence. Then $\rho_-(\lambda_n)\rightarrow\infty$ in $\rho_-(\I)$. In particular $\rho_-$ is not continuous.
\end{prop}

\begin{proof}
Let $A =\rho_{-}(\I)$ and $\s_+$ be the sector-germ of $A$ opposite to $\s_-$. Using Theorem~\ref{thmEquivalence des distances}, we may assume that $\theta=(|\ .\ |,\s_+)$. Let $\rho_+:\I\overset{\s_+}{\rightarrow} A$. Let $\lambda=\lim \lambda_n$ and $\mu=\rho_+(\lambda)$. Then by Corollary~\ref{corRetractions lipschitziennes}, $\rho_+(\lambda_n)\rightarrow \mu$. Let $Y_A=\I_0\cap A$. Then $Y_A$ is a lattice  in $A$ by Lemma~\ref{lemIntersection de I0 et A}. As $\rho_+(\lambda_n)\in Y_A$ for all $n\in \N$, $\rho_+(\lambda_n)=\mu$ for $n$ large enough.

 For all $n\in \N$, $\rho_-(\lambda_n)\in Y_A$. By Theorem 5.6 of \cite{hebertGK}, for all $\lambda'\in Y_A$, $\rho^{-1}_+(\{\lambda'\})\cap\rho_-^{-1}(\{\mu\})$ is finite, and the proposition follows. $\square$
\end{proof}

\begin{cor}\label{corNon egalite de T+ et T-}
If $\I$ is not a building, $\mathscr{T}_+$ and $\mathscr{T}_-$ are different.
\end{cor}

\begin{remark}\label{rqueConjecture}
Proposition~\ref{propNon egalite de T+ et T-} shows that if $\theta,\theta'\in \Theta$ have opposite signs, then  every open subset of $(\I,d_\theta)$ containing a point of $\I_0$ is unbounded with respect to $d_{\theta'}$. \end{remark}

\subsection{Mixed distances}\label{subSecDistances combinees}
In this section we define and study mixed distances.

 Let $\Xi=\Theta_+\times \Theta_-$. Let $\xi=(\theta_+,\theta_-)\in \Xi$. Set $d_{\xi}=d_{\theta_+}+d_{\theta_-}$.

\begin{theorem}\label{thmResume sur les Dtheta2}
Let $\xi\in \Xi$. We equip $\I$ with $d_\xi$. Then:
\begin{enumerate}
\item\label{itMixed_equivalence distances} For each $\xi'\in \Xi$, $d_\xi$ and $d_{\xi'}$ are equivalent.

\item\label{itMixed_action of G lipschitz } For each  $g\in G$, the induced map $g:\I\rightarrow \I$ is Lipschitz.

\item\label{itMixed_topology apartments} The topology induced on  every apartment is the affine topology.

\item\label{itMixed_retractions} Every retraction of $\I$ centered at a sector-germ is Lipschitz.

\item\label{itMixed I0 discrete} The set $\I_0$ is discrete.
\end{enumerate}
\end{theorem}

\begin{proof}
The assertions~\ref{itMixed_equivalence distances} to~\ref{itMixed_retractions} are consequences of  Theorem~\ref{thmEquivalence des distances}, Theorem~\ref{thmAction de G lipschitzienne}, Corollary~\ref{corTopologie induite sur les appartements} and Corollary~\ref{corRetractions lipschitziennes}. Let us prove (\ref{itMixed I0 discrete}). Let $\lambda\in \I_0$ and set $\lambda_+=\rho_{+\infty}(\lambda)$ and $\lambda_-=\rho_{-\infty}(\lambda)$. By Theorem~5.6 of \cite{hebertGK}, $\rho_{+\infty}^{-1}(\{\lambda_+\})\cap \rho_{-\infty}^{-1}(\{\lambda_-\})$ is finite and thus there exists $r>0$ such that $B(\lambda,r)\cap \rho_{+\infty}^{-1}(\{\lambda_+\})\cap \rho_{-\infty}^{-1}(\{\lambda_-\})=\{\lambda\}$, where $B(\lambda,r)$ is the open ball of the radius $r$ and the center $\lambda$.  Let $k\in\R_{>0}$ be such that $\rho_{+\infty}$ and $\rho_{-\infty}$ are $k$-Lipschitz. Let $\eta>0$ be such that for all $\mu,\mu'\in Y$, $\mu\neq \mu'$ implies $d_\xi(\mu,\mu')\geq \eta$. Let $r'=\min (r,\frac{\eta}{k})$. Let us prove that $B(\lambda,r')\cap \I_0=\{\lambda\}$. Let $\mu\in B(\lambda,r')\cap \I_0$.  Suppose $\rho_{\sigma\infty}(\mu)\neq \lambda_\sigma$, for some $\sigma\in \{-,+\}$. 
Then \[kd_\xi(\mu,\lambda)\geq d_\xi(\rho_{\sigma\infty}(\mu),\rho_{\sigma\infty}(\lambda)) \geq \eta,\] thus $\lambda\notin B(\lambda,r')$, a contradiction. Therefore $\rho_{+\infty}(\mu)=\lambda_+$ and $\rho_{-\infty}(\mu)=\lambda_-$,  hence $\lambda=\mu$ by choice of $r$, which completes  the proof of the theorem. $\square$
\end{proof}

We denote by $\mathscr{T}_m$ the topology on $\I$ induced by any $d_\xi$, $\xi\in \Xi$.

\subsection{Link with the initial topology with respect  to the retractions}

In this subsection, we prove that the topology $\mathscr{T}_m$ agrees with the initial topology with respect to the family of retractions centered at sector-germs (see Corollary~\ref{corDescription de la topologie combinee}). To this end, for each $u\in C^v_f$ we introduce a map $T_u:\I\rightarrow \R_+$ which, for each $x\in \I$, measures the distance along the ray $x+(\R_+ u)_\infty$ between $x\in \I$ and $\A$. We then use the fact that for all $\lambda\in Y\cap C^v_f$, $T_\lambda\leq \ell (\rho_{+\infty}-\rho_{-\infty})$, for some $\ell\in \R_+$ (see Lemma~\ref{propMajoration of Tnu}).

Fix a norm $|\ .\ |$ on $\A$.

\subsubsection*{Definition of $y_u$ and $T_u$}  We now review briefly the results of the paragraph ``Definition of $y_\nu$ and $T_\nu$'' of Section~3 of \cite{hebertGK}. Let $u\in C^v_f$ and $\sigma\in \{-,+\}$. Let $\delta_+=\R_+ u\subset \A$ and $\delta_-=\R_- u\subset \A$. Then $\delta_+$ and $\delta_-$ are generic rays. Let $x\in \I$, then   there exists a unique $y_{\sigma u}(x)\in \A$ such that $x+\delta_{\sigma,\infty}\cap \A=y_{\sigma u}(x)+\sigma \R_+ u\subset \A$ and there exists a unique $T_{\sigma u}(x)\in \R_+$ such that \[y_{\sigma u}(x)=x+_{\sigma \infty}T_{\sigma u}(x).u =\rho_{\sigma \infty}(x)+\sigma T_{\sigma u}(x). u.\] 

Then for each $x\in \I$, $x\in \A$ if and only if $y_u(x)=x$ if and only if $T_u(x)=0$. 

\begin{lemma}\label{propMajoration of Tnu}
Let $\lambda\in Y\cap C^v_f$. Then there exists $\ell_{|\ .\ |}\in \R_{>0}$ such that for all $x\in \I$,  \[T_ \lambda(x),T_{- \lambda}(x)\leq \ell_{|\ .\ |} |\rho_{+\infty}(x)-\rho_{-\infty}(x)|.\] 
\end{lemma}

\begin{proof}
By Corollary~4.2 and Remark~4.3 of \cite{hebertGK}, there exists a linear map $h:\A\rightarrow \R$ such that $T_ \lambda(x),T_{- \lambda}(x)\leq h\big(\rho_{-\infty}(x)-\rho_{+\infty}(x)\big)$ for all $x\in \I$, which proves the existence of $\ell_{|\ .\ |}$.
 $\square$
\end{proof}

\begin{lemma}\label{lemLien entre distance et retractions}
Let $\xi\in \Xi$ and $a\in \I$. Let $A$ be an apartment  containing $a$.  Let $\s_-,\s_+$ be two opposite sector-germs of $A$ and $\rho_+:\I\overset{\s_+}{\rightarrow} A$, $\rho_-:\I\overset{\s_-}{\rightarrow}A$. Then there exists $k\in \R_{>0}$ such that for all $x\in \I$, $d_\xi(a,x)\leq k\big(d_\xi\big(a,\rho_-(x)\big)+d_\xi\big(a,\rho_+(x)\big)\big)$. 
\end{lemma}

\begin{proof}
Using isomorphisms of apartments, we may assume $A=\A$, $\s_+=+\infty$ and $\s_-=-\infty$.  By Theorem~\ref{thmResume sur les Dtheta2}~(\ref{itMixed_equivalence distances}) we may assume $\xi=\big((|\ .\ |,\s_+),(|\ .\ |,\s_-)\big)$.

Let $ \lambda\in C^v_f$. Let $T_+=T_ \lambda:\I\rightarrow \R_+$
   and $T_-=T_{- \lambda}:\I\rightarrow \R_+$.  By Lemma~\ref{propMajoration of Tnu} and Lemma~\ref{lemDistances restricted to apartments containing s}~(\ref{itDtheta norme}), there exists $\ell\in \R_{>0}$ such that $T_\sigma(x)\leq \ell d_\xi(\rho_-(x),\rho_+(x))$ for all $x\in \I$ and both $\sigma\in \{-,+\}$.
   
Set $d_+=d_{(|\ .\ |,+\infty)}$ and $d_-=d_{(|\ .\ |,-\infty)}$.  Let $x\in \I$ and $\sigma\in \{-,+\}$. One has $x+_{\sigma\infty}T_\sigma(x) u=\rho_{\sigma}(x)+\sigma T_\sigma(x) u$. Thus \[\begin{aligned} d_\sigma(\rho_\sigma(x),x)\leq 2T_\sigma(x)| u| &\leq 2 \ell | u|d_\sigma(\rho_-(x),\rho_+(x))\\ &\leq 2\ell | u|\big(d_\sigma(\rho_-(x),a)+d_\sigma(\rho_+(x),a)\big).\end{aligned}\]

     As $d_\sigma(a,x)\leq d\big(a,\rho_\sigma(x)\big)+d\big(\rho_\sigma(x),x\big)$ we deduce that \[d_\xi(a,x)=d_-(a,x)+d_+(a,x)\leq (4\ell | u|+2)\big(d_\xi(a,\rho_-(x))+d_\xi(a,\rho_+(x))\big).\ \square\]

\end{proof}

\begin{cor}\label{propParties bornees pour Tn}
Let $\xi\in \Xi$.  Equip $\I$ with  $d=d_\xi$. Then, for $X\subset \I$ the following assertions are equivalent:

\begin{enumerate}
\item\label{itBounded set} X is bounded.
\item\label{itBounded retracted sets} For every retraction $\rho$ centered at a sector-germ of $\I$, $\rho(X)$ is bounded.
\item\label{itBounded retracted sets for two retractions} There exist two opposite sector-germs $\s_+$ and $\s_-$ such that if $\rho_{\s_-}$ and $\rho_{\s_+}$ are retractions centered at $\s_-$ and $\s_+$, $\rho_{\s_-}(X)$ and $\rho_{\s_+}(X)$ are bounded.
\end{enumerate}

Moreover every bounded subset of $\I_0$ is finite.
\end{cor}

\begin{proof}
By Theorem~\ref{thmResume sur les Dtheta2}, (\ref{itBounded set}) implies (\ref{itBounded retracted sets}) which clearly  implies (\ref{itBounded retracted sets for two retractions}). The fact that (\ref{itBounded retracted sets for two retractions}) implies (\ref{itBounded set}) is a consequence of Lemma~\ref{lemLien entre distance et retractions}. The last assertion is a consequence of (\ref{itBounded retracted sets for two retractions}) and of Theorem 5.6 of \cite{hebertGK}. $\square$
\end{proof}

\begin{cor}\label{corDescription de la topologie combinee}
The sets $\rho_+^{-1}(V)\cap \rho_{-}^{-1}(V)$ such that $V$ is an open set of an apartment $A$ and $\rho_{-}$, $\rho_{+}$ are retractions onto $A$  centered at opposite sector-germs of $A$ form a basis of $\mathscr{T}_m$. In particular $\mathscr{T}_m$ is the initial topology with respect to the retractions centered at sector-germs.
\end{cor}

\begin{proof}
This is a consequence of Lemma~\ref{lemLien entre distance et retractions}. $\square$
\end{proof}

\subsection{A continuity property for the map $+_{+\infty}$}

 The aim of this subsection is to prove the  theorem below, which will be useful to prove the contractibility of $\I$ for $\mathscr{T}_m$. To simplify the notation, we write $+$ instead of $+_{+\infty}$.

\begin{theorem}\label{lemContinuity of + for Tm}
Let $\xi\in \Xi$ and $u\in C^v_f$. Equip $\I$ with $d_\xi$. Then the map $\I\times \R_+\rightarrow \I $ defined by $(x,t)\mapsto x+tu$ is continuous.
\end{theorem}

To prove this theorem we prove that if a sequence $\big((x_n),(t_n)\big)\in (\I\times \R_+)^\N$ converges to some $(x,t)\in \I\times \R_+$, then $(x_n+t_nu)$ converges to $x+tu$. We first treat the case where $t\neq 0$. 

Fix $\xi\in \Xi$ and write $\xi=(\theta_+,\theta_-)$. Fix a norm $|\ .\ |$ on $\A$.

\begin{lemma}\label{lemLipschitziannite des T_nu et y_nu}
Let $u\in C^v_f$.  Then $T_u:\I\rightarrow \R_+$ and $y_u:\I\rightarrow (\A,|\ .\ |)$ are Lipschitz for $d_{\theta_+}$ and $d_\xi$.
\end{lemma}

\begin{proof}
By Theorem~\ref{thmEquivalence des distances} and Theorem~\ref{thmResume sur les Dtheta2}, we can assume $\theta_+=(+\infty,|\ .\ |)$. Let $\ell\in \R_{>0}$ be such that for all $a\in \A$, $\ell |a| u-a\in C^v_f$. Let $x,x'\in \I$ and $(u,u')\in U_{+\infty}(x,x')$ be such that $d_{\theta_+}(x,x')= |u|+|u'|$, which exists by Lemma~\ref{rqueDistance atteinte}. Then $x+T_ u(x) u\in \A$ and thus $x'+u'+ T_ u(x) u=x+u+T_ u(x) u\in \A$. Therefore \[x'+u'+T_ u(x) u+(\ell |u'| u -u')=x'+(T_ u(x)+\ell |u'|) u \in \A.\] Hence $T_ u(x')\leq T_ u(x)+\ell |u'|\leq T_ u(x)+\ell  d_{\theta_+}(x,x')$. By symmetry we deduce that $T_ u$ is $\ell$-Lipschitz for $d_{\theta_+}$. The fact that $y_ u$ is Lipschitz for $d_{\theta_+}$ is a consequence of the continuity of the map $+$ (Lemma~\ref{lemLipschitziannite de l'addition}) and of the fact that $y_ u=\rho_{+\infty}+T_ u.  u$. As $d_{\theta_+}\leq d_\xi$, the lemma is proved. $\square$
\end{proof}

\begin{lemma}\label{lemContinuity of addition for Tm and tneq0}
 Let $u\in C^v_f$ and $(x_n,t_n)\in \I\times \R_{>0}$ be such that $x_n\rightarrow x$ for $d_\xi$ and $t_n\rightarrow t$, for some $x\in \I$ and $t\in \R_{>0}$. Then $x_n+t_n  u \rightarrow x+t u$ for $d_\xi$.
\end{lemma}

\begin{proof}
 First assume $x\in \A$. By Lemma~\ref{lemLipschitziannite des T_nu et y_nu}, $T_ u(x_n)\rightarrow T_ u(x)=0$. Consequently, for $n\in \N$ large enough, $x_n+t_n u\in \A$. Write $\xi=(\theta_+,\theta_-)$. By the continuity of the map $+$ for $d_{\theta_+}$ (Lemma~\ref{lemLipschitziannite de l'addition}), $x_n+t_n  u\rightarrow x+t u$ for $d_{\theta_+}$. As the topologies induced by $d_{\theta_+}$ and $d_\xi$ on $\A$ agree with the topology induced by its structure of a finite-dimensional real vector-space (by Corollary~\ref{corTopologie induite sur les appartements} and Theorem~\ref{thmResume sur les Dtheta2}~(\ref{itMixed_topology apartments})), we deduce that $x_n+ u\rightarrow x+ u$ for $d_\xi$. 

We no longer assume that $x\in \A$. Let $A$ be an apartment containing $x+\infty$. Let $\phi:A\overset{+\infty}{\rightarrow}\A$. Let $g\in G$ be an automorphism inducing $\phi$. By Lemma~\ref{lemAction of G on addition}, $x_n+t_n u=g^{-1}.\big(g.(x_n+t_n  u)\big)=g^{-1}.(g.x_n+ t_n u)$ for all $n\in \N$. As $g.x\in \A$, we deduce that $g.x_n+t_n u \rightarrow g.x+t u$ for $d_\xi$. By the  continuity of $g^{-1}:(\I,d_\xi)\rightarrow (\I,d_\xi)$ (by Theorem~\ref{thmResume sur les Dtheta2}~(\ref{itMixed_action of G lipschitz })), $x_n+ u\rightarrow x+ u$ for $d_\xi$. $\square$
\end{proof}

It remains to prove that if a sequence $(x_n,t_n)\in (\I\times \R_+)$ converges to $(x,0)$, for some $x\in \I$, then $(x_n+t_nu)$ converges to $x$. In order to prove this we first study the map $t\mapsto \rho_{-\infty}(x+tu)$.

\subsubsection*{Tits preorder on $\I$, vectorial distance on $\I$ and paths} Recall the definition of the Tits preorder $\leq$ on $\A$ from Subsection~\ref{subVectorial faces and Tits preorder}. As $\leq$ is invariant under the action of the Weyl group $W^v$, $\leq$ induces a preorder $\leq_A$ on every apartment $A$. Let $A$ be an apartment and $x,y\in A$ be such that $x\leq_A y$. Then by Proposition~5.4 of \cite{rousseau2011masures}, if $A'$ is an apartment containing $x,y$, $x\leq_{A'} y$. This enables to define the following relation $\leq$ on $\I$: if $x,y\in \I$, one says that $x\leq y$ if there exists an apartment $A$ containing $x,y$ and such that $x\leq_{A} y$. By Th{\'e}or{\`e}me~5.9 of \cite{rousseau2011masures}, this defines a preorder on $\I$ and one calls $\leq$ the \textit{Tits preorder}.

Let $x,x'\in \I$ be such that $x\leq x'$. Let $A$ be an apartment containing $x,x'$ and $f:A\rightarrow \A$ be an isomorphism of apartments. Then $f(x')-f(x)$ is in the Tits cone $\T$. Therefore there exists a unique $d^v(x,x')$ in $\overline{C^v_f}\cap W^v.\big(f(x')-f(x)\big)$. One calls $d^v$ the \textit{vectorial distance}.

Let $u\in \overline{C^v_f}$. A $u$-path is a piecewise linear continuous map $\pi:[0,1]\rightarrow \A$ such that each (existing) tangent vector $\pi'(t)$ belongs to $W^v.u$. Let $x,x'\in \I$ be such that $x\leq x'$, $A$ be an apartment containing them and $f:\A\rightarrow A$. We define $\pi_{x,x'}:[0,1]\rightarrow \I$ by $t\mapsto f\big((1-t)f^{-1}(x)+tf^{-1}(x)\big)$. By Proposition~5.4 of \cite{rousseau2011masures}, $\pi_{x,x'}$ does not depend on the choice of $A$.

Let $x\in \I$ and $u\in \overline{C^v_f}$. Then $d^v(x,x+u)=u$.

\begin{lemma}\label{lemImage of a segment by rho}
Let $x\in \I$ and $u\in \overline{C^v_f}$. Then $\rho_{-\infty}\circ\pi_{x,x+u}$ is a $u$-path.
\end{lemma}

\begin{proof}
This is a weak version of Theorem~6.2 of \cite{gaussent2008kac} (a Hecke path of the shape $u$ is a $u$-path satisfying some conditions, see Section~5 of \cite{gaussent2008kac} for the definition). $\square$
\end{proof}

Recall that the $\alpha_i^\vee$, for $i\in I$, denote the simple roots. Let $Q^\vee_{\R_+}=\{\sum_{i\in I} x_i\alpha_i^\vee|(x_i)\in (\R_+)^I\}\subset \A$.

\begin{lemma}\label{lemUpaths increasing for Q}
Let $u\in \overline{C^v_f}$ and $\pi:[0,1]\rightarrow \A$ be a $u$-path. Then $\pi(1)-\pi(0)-u\in Q^\vee_{\R_+}$.
\end{lemma}

\begin{proof}
Let $w\in W^v$. Then by Proposition 3.12 d) of \cite{kac1994infinite}, $w.u-u\in -Q^\vee_{\R_+}$. Thus for all $t$ such that $\pi'(t)$ is defined, $\pi'(t)-u\in -Q^\vee_{\R_+}$ and the lemma follows. $\square$
\end{proof}

 Let $|\ .\ |_0$ be a norm on $\A$ such that for all $q=\sum_{i\in i} q_i \alpha_i^\vee\in \bigoplus_{i\in I} \R\alpha_i^\vee$, $|q|_0=\sum_{i\in I} |q_i|$.

\begin{lemma}\label{lemMajoration of rho(x+tu) by rho(x+t'u)}
Let $x\in \I$, $u\in \overline{C^v_f}$ and $t,t'\in \R_+$ be such that $t\leq t'$. Then \[|\rho_{-\infty}(x+tu)-\rho_{-\infty}(x)|_0\leq (t+t')|u|_0+ |\rho_{-\infty}(x+t'u)-\rho_{-\infty}(x)|_0.\]
\end{lemma}

\begin{proof}
Write $\rho_{-\infty}(x+tu)-\rho_{-\infty}(x)=tu-q_1$ and $\rho_{-\infty}(x+t'u)-\rho_{-\infty}(x+tu)=(t'-t)u-q_2$, with $q_1,q_2\in Q^\vee_{\R_+}$, which is possible by Lemma~\ref{lemImage of a segment by rho} and Lemma~\ref{lemUpaths increasing for Q}. Then $\rho_{-\infty}(x+t'u)-\rho_{-\infty}(x)=t'u-q_1-q_2$. One has $|\rho_{-\infty}(x+tu)-\rho_{-\infty}(x)|_0\leq t|u|_0+|q_1|_0$. By choice of $|\ .\ |_0$, $|q_1|_0\leq |q_1+q_2|_0=|\rho_{-\infty}(x+t'u)-\rho_{-\infty}(x)-t'u|_0$, and the lemma follows. $\square$
\end{proof}

The following lemma completes the proof of Theorem~\ref{lemContinuity of + for Tm}.

\begin{lemma}
Let $u\in C^v_f$. Let $(x_n)\in \I^\N$ and $(t_n)\in \R_+^\N$ be such that $(x_n)$ converges for $d_\xi$ and $(t_n)$ converges to  $0$. Then $(x_n+t_n u)$ converges to  $\lim x_n$ for $d_\xi$.
\end{lemma}

\begin{proof}
By Theorem~\ref{thmResume sur les Dtheta2}, we can assume $\xi=\big( (|\ .\ |_0,+\infty),(|\ .\ |_0,-\infty)$.  Let $x=\lim x_n$. By the same reasoning as in the proof of Lemma~\ref{lemContinuity of addition for Tm and tneq0}: \begin{itemize}
\item[-] we can assume $x\in \A$,

\item[-] $x_n+(T_u(x_n)+t_n)u\rightarrow x$.\end{itemize} 

By Lemma~\ref{lemMajoration of rho(x+tu) by rho(x+t'u)}, for all $n\in \N$, \[|\rho_{-\infty}(x_n+t_nu)-\rho_{-\infty}(x_n)|_0\leq |\rho_{-\infty}(x_n+(T_u(x_n)+t_n)u)-\rho_{-\infty}(x)|_0+(T_u(x_n)+2t_n)|u|_0,\] and thus $\rho_{-\infty}(x_n+t_n u)\rightarrow \rho_{-\infty}(x)$.

By continuity of the map $+$ (Lemma~\ref{lemLipschitziannite de l'addition}) and the continuity of $\rho_{+\infty}$ (Corollary~\ref{corRetractions lipschitziennes}) for  $d_{\theta_+}$, $\rho_{+\infty}(x_n+t_nu)\rightarrow \rho_{+\infty}(x)$. Using Lemma~\ref{lemLien entre distance et retractions} we deduce that $(x_n+t_nu)$ converges to  $x$, which is the desired conclusion. $\square$
\end{proof}

\section{Contractibility of $\I$}\label{secContractilite de I}
In this section we prove the contractibility of $\I$ for $\mathscr{T}_+$, $\mathscr{T}_-$ and $\mathscr{T}_m$.

Let $|\ .\ |$ be a norm on $\A$, $\theta=(|\ .\ |,+\infty)$ and $\xi=\big( (|\ .\ |,+\infty),(|\ .\ |,-\infty)\big)$. To simplify the notation we write $+$ instead of $+_{+\infty}$.

\begin{prop}\label{propI se retracte fortement transitivement sur A}
Let $u\in C^v_f$. We define $\chi_u:\I\times [0,1]\rightarrow \I$ by \[\left\{\begin{aligned} & \chi_u(x,t)= x+\frac{t}{1-t}u &\mathrm{\ if\ } \frac{t}{1-t}<T_u(x)\\
&  \chi_u(x,t)= y_u(x) &\mathrm{\ if\ } \frac{t}{1-t}\geq T_u(x),\end{aligned}\right.\] where we set $\frac{1}{0}=+\infty >t$ for all $t\in \R$. Then $\chi_u$ is a strong deformation retract on $\A$ for $d_\theta$ and $d_\xi$.
\end{prop}

\begin{proof}
Let $x\in \A$ and $t\in [0,1]$. Then $T_u(x)=0$ and  thus $\chi_u(x,t)=y_u(x)=x$. 
Let $x\in \I$. Then $\chi_u(x,0)=x$ and $\chi_u(x,1)=y_u(x)\in \A$. It remains to show that $\chi_u$ is continuous for $d_\theta$ and $d_\xi$. Let $(x_n,t_n)\in (\I\times [0,1])^\N$ be a converging sequence for $d_\theta$ or $d_\xi$ and $(x,t)=\lim (x_n,t_n)$. Suppose for example that $\frac{t}{1-t}<T_u(x)$ (the case $\frac{t}{1-t}=T_u(x)$ and $\frac{t}{1-t}>T_u(x)$  are analogous). Then by the continuity of $T_u$ (Lemma~\ref{lemLipschitziannite des T_nu et y_nu}), $\frac{t_n}{1-t_n}<T_u(x_n)$ for $n$ large enough and thus by the continuity of the map $+$ (Lemma~\ref{lemLipschitziannite de l'addition} for $d_\theta$ and Theorem~\ref{lemContinuity of + for Tm} for $d_\xi$), $\chi_u(x_n,t_n)=x_n+\frac{t_n}{1-t_n}u\rightarrow x+\frac{t}{1-t}u=\chi_u(x,t)$. Therefore, $\chi_u$ is continuous, which concludes the proof. $\square$
\end{proof}

\begin{cor}\label{corContractilite de I}
The masure $\I$ is contractible for $\mathscr{T}_+$, $\mathscr{T}_-$ and $\mathscr{T}_m$.
\end{cor}

\begin{proof}
Let $u\in C^v_f$. We define $\Upsilon_u:\I\times [0,1]\rightarrow \I$ by \[\left\{\begin{aligned} &\Upsilon_u(x,t) = \chi_u(x,2t) &\mathrm{\ if\ }t\leq \frac{1}{2}\\ &\Upsilon_u(x,t) = 2(1-t)y_u(x) &\mathrm{\ if\ t>\frac{1}{2}}.\end{aligned}\right.\] Then $\Upsilon_u$ is a strong deformation retract on $\{0\}$ for $d_\theta$ and $d_\xi$, which proves that $(\I,\mathscr{T}_+)$ and $(\I,\mathscr{T}_m)$ are contractible. By symmetry, $(\I,\mathscr{T}_-)$ is contractible.  $\square$
\end{proof}

\bibliographystyle{plain}

\end{document}